\documentclass[a4paper]{article}

\usepackage[english]{babel}

\usepackage{cite}

\usepackage[utf8]{inputenc}
\usepackage[T1]{fontenc}
\usepackage{lmodern}

\usepackage{amssymb,amsmath,amsthm}

\usepackage{todonotes}
\usepackage{mathtools,bbm}
\usepackage[normalem]{ulem}
\usepackage{cancel}
\usepackage{verbatim}

\usepackage{enumerate}
\usepackage[shortlabels]{enumitem}

\usepackage{a4}
\usepackage{multirow}
\usepackage[footnotesize,bf,centerlast]{caption}
\usepackage[small,euler-digits,icomma,OT1,T1]{eulervm}

\usepackage{xcolor,colortbl}
\definecolor{Gray}{gray}{0.80}
\definecolor{LightGray}{gray}{0.90}

\usepackage{hyperref}

\newcommand{\cA}{\mathcal{A}}
\newcommand{\cB}{\mathcal{B}}
\newcommand{\cC}{\mathcal{C}}
\newcommand{\cD}{\mathcal{D}}

\newcommand{\cG}{\mathcal{G}}
\newcommand{\cH}{\mathcal{H}}

\newcommand{\cJ}{\mathcal{J}}

\newcommand{\cL}{\mathcal{L}}
\newcommand{\cM}{\mathcal{M}}

\newcommand{\cO}{\mathcal{O}}
\newcommand{\cP}{\mathcal{P}}

\newcommand{\cU}{\mathcal{U}}
\newcommand{\cV}{\mathcal{V}}

\newcommand{\bE}{\mathbb{E}}

\newcommand{\bH}{\mathbb{H}}

\newcommand{\bN}{\mathbb{N}}

\newcommand{\bR}{\mathbb{R}}

\newcommand{\bZ}{\mathbb{Z}}

\newcommand{\bfH}{\mathbf{H}}

\newcommand{\bfR}{\mathbf{R}}

\newcommand{\bfV}{\mathbf{V}}

\newcommand{\PR}{\mathbb{P}}
\newcommand{\bONE}{\mathbbm{1}}

\newcommand{\dd}{ \mathrm{d}}

\DeclareMathOperator*{\LIM}{LIM} 
\DeclareMathOperator*{\subLIM}{subLIM}
\DeclareMathOperator*{\superLIM}{superLIM}

\renewcommand{\epsilon}{\varepsilon}

\newcommand{\vn}[1]{\left| \! \left| #1\right| \! \right|}

\newcommand{\ip}[2]{\langle #1,#2\rangle}

\numberwithin{equation}{section}

\newtheorem{theorem}{Theorem}[section]
\newtheorem{lemma}[theorem]{Lemma}
\newtheorem{proposition}[theorem]{Proposition}

\theoremstyle{definition}
\newtheorem{definition}[theorem]{Definition}

\newtheorem{remark}[theorem]{Remark}

\newtheorem{assumption}[theorem]{Assumption}
\newtheorem{example}[theorem]{Example}
\newtheorem{condition}[theorem]{Condition}

\title{Well-posedness of Hamilton-Jacobi equations in population dynamics and applications to large deviations}

\author{Richard C. Kraaij\thanks{Department of Applied Mathematics, Delft University of Technology, The Netherlands. \emph{E-mail address}: R.C.Kraaij@tudelft.nl} \and Louis Mahé \thanks{CPHT, Ecole Polytechnique, CNRS, Université Paris-Saclay, Route de Saclay, 91128 Palaiseau, France \emph{E-mail address}: louis.mahe@polytechnique.edu }}
\date{\today}

\begin{document}

\maketitle

\begin{abstract}
	We prove Freidlin-Wentzell type large deviation principles for various rescaled models in populations dynamics that have immigration and possibly harvesting: birth-death processes, Galton-Watson trees, epidemic SI models, and prey-predator models. 
	
	The proofs are carried out using a general analytic approach based on the well-posedness of a class of associated Hamilton-Jacobi equations. The notable feature for these Hamilton-Jacobi equations is that the Hamiltonian can be discontinuous at the boundary.
	We prove a well-posedness result for a large class of Hamilton-Jacobi equations corresponding to one-dimensional models, and give partial results for the multi-dimensional setting. 
	

\end{abstract}


\section{Introduction}

The evolution of the population size of one or more interacting species can be modeled using Markov jump processes. The jumps correspond to the result of interactions between and amongst individuals of the population: predation, reproduction or correspond to global phenomena such as immigration, emigration, harvesting, or mutation (see e.g. \cite{Al10,MuKu15}). 

A much studied question is the dynamics of the population in the large population scaling limit. Typical results include laws of large numbers to a solution of a first-order differential equation or weak convergence of the process to a solution of a stochastic differential equation as in \cite{BaMe15,MuKu15,CCLS17} and \cite[Chapter 11]{EK86}.

Additional work is required for the estimation of the exit time from a domain: time spent before extinction, time for the propagation of a disease among a population, time for a mutation to establish itself in a population. This motivates the study of path-space large deviations of Freidlin-Wentzell type \cite{FW98}. In short, given a sequence of stochastic processes $X_n(t)$, one aims to establish
\begin{equation*}
\PR[\{X_n(t)\}_{t \geq 0} \approx \{\gamma(t)\}_{t \geq 0}] \approx e^{-n I(\gamma)},
\end{equation*}
where 
\begin{equation*}
(\gamma) = \begin{cases}
I_0(\gamma(0)) + \int_0^\infty \cL(\gamma(s),\dot{\gamma}(s)) \dd s & \text{if } \gamma \in \cA\cC, \\
\infty & \text{otherwise}.
\end{cases} 
\end{equation*}
Here $\cA\cC$ is an appropriate space of absolutely continuous trajectories and where the \textit{Lagrangian} $\cL$ is some non-negative function.

A notable feature of the large deviations for models in population dynamics, as well as in the analysis of queuing systems, is that the systems have boundaries (arising e.g. from the condition that a population size is non-negative) and that the jump rates vanish at the boundary. Large deviations for such systems have been considered by \cite{Ch97,KrPa16,PaSa17} in the context of population dynamics and in \cite{ShWe05} in the context of queuing systems. Indeed, it is the boundary where the rates vanish, that leads to issues in adapting the methods of \cite{FW98}. 

To obtain results, systems are assumed to start in the interior, and the large deviation principle is proven by approximating trajectories that hit the boundary by trajectories that remain in the interior. In addition, the rate function evaluated in the approximating trajectories need to approximate the rate function evaluated in the limiting trajectory. This approach has two major drawbacks:
\begin{enumerate}[(a)]
	\item the pre-limit processes are not allowed to start on the boundary,
	\item the approach does not allow for discontinuous rates at the boundary.
\end{enumerate}
These problems were also identified by \cite{DuIiSo90} in the context of queuing systems in which the rates are discontinuous for natural reasons: if there are customers in the queue, a server is active whereas it is inactive if the queue is empty. This leads to the analysis of a system that has homogeneous rates in the interior of the space, but with discontinuity at the boundary. To treat the large deviations of such systems the authors instead turn to characterizing the large deviations via solutions to a Hamilton-Jacobi equation and prove well-posedness for the Hamilton-Jacobi equation. See also \cite{DuEl92} for a model with discontinuities in the interior instead of at the boundary.

In this paper, we study the large deviations for population dynamics, including the possibility of starting at the boundary as well as discontinuous behaviour at the boundary. This leads to a set of models in which we have inhomogeneous rates in the interior that might vanish at the boundary or have discontinuities at the boundary. 

Discontinuities appear naturally e.g. in the context of population dynamics in the presence of harvesting or predation. A (short-sighted) external entity takes out individuals from the population of the system at a fixed rate, e.g. predators hunting prey possibly until extinction, or the influence of the fishing industry on the population of fish in a lake or sea. Clearly, no population can be taken out of the system when there is no population left.

Our aim is to prove the large deviation principle {for jump processes} that can treat discontinuities as well as vanishing rates. Our main result complements the established results mentioned above by covering the middle ground. As an added benefit, our approach partially extends \cite{KrPa16,PaSa17} by allowing for non-Lipschitz rates. In particular, this means we can consider models that do not have a deterministic limit (See e.g. \cite[Example 2.5.3]{Kr16b}.) We do mention that we cannot allow for an as general state-space as in the recent paper \cite{KrPa16}. Neither have we pushed our results for multi-dimensional models into the domain of \cite{DuIiSo90} that were treated with non-quadratic test functions. We further discuss the possibilities and restrictions of our approach in Section \ref{section:extendability}.
	
Even though this is not the focus of our paper, the results can be used to study the exit time from a domain. In the context of biological processes starting at the boundary, this allows e.g. to study the time that a new species needs to establish itself in a hostile new environment.

\smallskip

We follow \cite{DuIiSo90} in studying large deviations via the theory of viscosity solutions to Hamilton-Jacobi equations. However, we rather follow the more recent framework introduced by \cite{FK06}, who essentially reduce proving the large deviation principle to a well-posedness question to a class of Hamilton-Jacobi equations. This method has been used recently in various papers, see e.g. \cite{DFL11, CoKr17,KrReVe19} and references therein, in which various steps in this method have been explained in detail.

The main technical step in this paper is, therefore, the establishment of well-posedness for the associated Hamilton-Jacobi equations. To be able to treat discontinuities at the boundary in our specific context, we also introduce a framework to account for which types of behaviour are exhibited by the process close to the boundary. Indeed, our Lagrangian needs to take into account the various types of possible behaviour as in \cite{DuIiSo90}. Finally, we argue that there are settings in which the large deviation principle as well as the well-posedness of the Hamilton-Jacobi equation fails to hold. 

\smallskip

The rest of the paper is organized as follows. We start in Section \ref{section:preliminaries} with general preliminaries. In addition, we introduce a framework to describe the exact structure of our space and its boundaries. We then introduce collections of Hamiltonians that take into account the various types of behaviour exhibited by the processes close to the boundary.
Next, in Section \ref{section:LDP_for_examples} we proceed by applying our main abstract results to a collection of examples to illustrate the range of applications. We treat various one-dimesional models including birth-death processes, processes with births only like the process that models the size of a Galton-Watson tree and of some simple epidemic models. We then discuss the failure of well-posedness for a class of birth-death processes. We conclude the section with a treatment of a class of multi-dimensional processes in which there is immigration and harvesting for each species.

We proceed with the comparison principle for Hamilton-Jacobi equations on spaces with boundaries in Section \ref{section:framework_HJ_boundary} and to apply these results to the theory of large deviations in Section \ref{section:abstract_LDP}. We conclude our paper with Section \ref{section:proofs_examples} in which we apply the results of Sections \ref{section:framework_HJ_boundary} and \ref{section:abstract_LDP} to the examples of Section \ref{section:LDP_for_examples}.

\bigskip

\section{Preliminaries} \label{section:preliminaries}

Before starting with our main results, we state some preliminaries  needed for the formulation of our main results. In Section \ref{section:general_preliminaries}, we give some general mathematical preliminaries including the main definitions of large deviation theory. We proceed in Section \ref{section:boundary_structure} by introducing the class of spaces on which our results will be applicable, in addition we introduce a suitable framework to effectively study the behavior of our Hamiltonians on the boundaries of our spaces. Finally, in Section \ref{section:generating_sets_of_Hamiltonians} we introduce a class of Hamiltonians and Lagrangians which are suited to for our spaces with boundary.

\subsection{General preliminaries} \label{section:general_preliminaries}

Let $E$ be a $d$-dimensional convex polyhedron with non-empty interior. A full definition follows on page \pageref{definition:base_space} below. We denote by $\cP(E)$ the space of Borel probability measures on $E$ and we denote by $D_E(\bR^+)$ the space of paths $\gamma : \bR^+ \rightarrow E$ that are right continuous and have left limits. We endow $D_E(\bR^+)$ with the Skorokhod topology, cf. \cite[Section 3.5]{EK86}. An important property is that under this topology $D_E(\bR^+)$ is Polish if $E$ is Polish.

We denote by $C_b(E), C_l(E)$ and $C_u(E)$ the spaces of continuous functions that are bounded, bounded from below and bounded from above respectively. We write $C_c(E)$ for the functions with compact support in $E$ relative to the topology of $E$. Finally, we denote $C_c^k(E)$ for functions in $C_c(E)$ that are $k$ times continuously differentiable.

Let $\overline{\bR} = [-\infty,\infty]$ be the two-point compactification of $\bR$. We denote by $\cM(E,\overline{\bR})$ the space of measurable functions from $E$ into $\overline{\bR}$. The spaces $\cM_b(E,\overline{\bR})$, $\cM_l(E,\overline{\bR})$ and $\cM_u(E,\overline{\bR})$ are the subspaces of $\cM(E,\overline{\bR})$ with functions that are bounded, bounded from above and below from respectively.

We say that a set $F \subseteq C_b(E)$ \textit{separates points} in $E$ if for all $x,y \in E$, $x\neq y$ there is $f \in F$ with $f(x) \neq f(y)$.

For a function $f : E \rightarrow \bR$, we denote by $f^*$ and $f_*$ the upper semi-continuous and lower semi-continuous regularization respectively. 

For a set $G \subseteq E$, we write $ch G$ for the closed convex hull of $G$ in $E$.

\begin{definition}
	Let $\{X_n\}_{n \geq 1}$ be a sequence of random variables on a Polish space $\mathcal{X}$. Furthermore, consider a function $I : \mathcal{X} \rightarrow [0,\infty]$ and a sequence $\{r_n\}_{n \geq 1}$ of positive numbers such that $r_n \rightarrow \infty$. We say that
	\begin{itemize}
		\item $I$ is a \textit{good rate-function} if the set $\{x \, | \, I(x) \leq c\}$ is compact for every $c \geq 0$.
		\item 
		the sequence $\{X_n\}_{n\geq 1}$ is \textit{exponentially tight} at speed $r_n$ if, for every $a \geq 0$, there exists a compact set $K_a \subseteq \mathcal{X}$ such that $\limsup_n r_n^{-1} \log \, \PR[X_n \notin K_a] \leq - a$.
		\item 
		the sequence $\{X_n\}_{n\geq 1}$ satisfies the \textit{large deviation principle} with speed $r_n$ and good rate-function $I$, denoted by 
		\begin{equation*}
		\PR[X_n \approx a] \asymp e^{-r_n I(a)},
		\end{equation*}
		if, for every closed set $A \subseteq \mathcal{X}$, we have 
		\begin{equation*}
		\limsup_{n \rightarrow \infty} \, r_n^{-1} \log \PR[X_n \in A] \leq - \inf_{x \in A} I(x),
		\end{equation*}
		and, for every open set $U \subseteq \mathcal{X}$, 
		\begin{equation*}
		\liminf_{n \rightarrow \infty} \, r_n^{-1} \log \PR[X_n \in U] \geq - \inf_{x \in U} I(x).
		\end{equation*}
	\end{itemize}
\end{definition}

We denote $\cA\cC(E)$ for the set of absolutely continuous curves in $E$. For the sake of completeness, we recall the definition of absolute continuity.

\begin{definition} 
	A curve $\gamma: [0,T] \to E$ is absolutely continuous in $E$ if there exists a function $g \in L^1([0,T],\bR^d)$ such that for $t \in [0,T]$ we have $\gamma(t) = \gamma(0) + \int_0^t g(s) \dd s$. We write $g = \dot{\gamma}$. A curve $\gamma: \bR^+ \to E$ is absolutely continuous in $E$ if the restriction to $[0,T]$ is absolutely continuous for every $T \geq 0$. 	
\end{definition}

Finally, we write $\bN = \{0,1, \dots\}$ for the set of natural numbers including $0$.

\subsection{Boundary structure of our state-spaces} \label{section:boundary_structure}

For $x,v \in \bR^d$, we write
\begin{equation*}
\cB_{x,v} := \left\{x + w \, \middle| \, w\in \bR^d, \ip{w}{v} \geq 0  \right\}, \quad \cB_{x,v}^\circ := \left\{x + w \, \middle| \, w\in \bR^d, \ip{w}{v} > 0  \right\},
\end{equation*}
for the closed and open half-spaces based at $x$ in the direction $v$.

\begin{definition}[$d$ dimensional convex Polyhedron]\label{definition:base_space}
	We will say that $E$ is a \textit{$d$ dimensional convex Polyhedron} if there are $x_1,\dots,x_k, y_1,\dots,y_l \in \bR^d$ $v_1,\dots,v_k, w_1,\dots,w_l \in \bR^d$, $v_i \neq v_j, v_i \neq w_j, w_i \neq w_j$  such that
	\begin{equation} \label{eqn:definition_intersection_E}
	E := \left(\bigcap_{i =1}^k \cB_{x_i,v_i} \right) \cap \left(\bigcap_{j =1}^l \cB_{y_j,w_j}^\circ\right). 
	\end{equation}
	Without loss of generality, we assume that the interior of $E$ is non-empty. (Otherwise we perform a coordinate transform and work in a lower-dimensional space). 
\end{definition}

\begin{definition}[Tangent spaces to a $d$ dimensional convex Polyhedron] \label{definition:tangent_cones}
	Let $E$ be a $d$ dimensional convex Polyhedron. 
	\begin{enumerate}[(a)]
		\item Let $\cJ = \cJ(E)$ be the set of subsets $J \subseteq \{1,\dots,k\}$ such that
		\begin{equation*}
		E_J := E \cap \bigcap_{i \in J} \left(\cB_{x_i,v_i} \setminus \cB_{x_i,v_i}^\circ\right) \neq \emptyset.
		\end{equation*}
		$E_J$ is a $|J|$ dimensional face of $E$.
		\item For $J \in \cJ$ such that $J \neq \emptyset$ write $\Gamma_J := \bigcap_{i \in J} \cB_{0,v_i}$, $\Gamma_\emptyset := \bR^d$. 
		\item For each $x \in E$ we define the set $T_x E := \bigcap_{J \in \cJ: \, x \in E_J} \Gamma_J$. 
		\item We denote by $J^*(x)$ the set of largest cardinality $J$ such that $x \in E_J$. Note that $T_x E = \Gamma_{J^*(x)}$.
		\item We write $TE = \left\{(x,v) \in E \times \bR^d \, \middle| \, v \in T_x E \right\}$. We equip $TE$ with the subspace topology inherited from $E \times \bR^d$.
	\end{enumerate}
\end{definition}

For each point $x \in E$, the space $T_x E$ is the set of `tangent vectors' at $x$: if $v \in T_x E$ then we have for small $\lambda > 0$ that $x + \lambda v \in E$.

We give four examples to clarify the definition.

\begin{example}[Half line] \label{example:half_line}
	Let $E = [0,\infty)$. Then for $(x_1,v_1) = (0,1)$, we have $E = B_{0,1}$. It follows that $\cJ = \{\emptyset, \{1\}\}$, $E_\emptyset = E$, $\Gamma_\emptyset = \bR$, $E_{\{1\}} = \{0\}$, $E_{\{1\}} = \{0\}$ and $\Gamma_{\{1\}} = [0,\infty)$. The tangent spaces are given by $T_0 E = \Gamma_{\{1\}} = [0,\infty)$ and $T_x E = \Gamma_\emptyset = \bR$.
\end{example}

\begin{example}[Half open half closed interval]
	Consider $E = [0,1)$. Then for $(x_1,v_1) =(0,1), (y_1,w_1) = (1,-1)$, we have $E = \cB_{0,1} \cap \cB_{1,-1}^\circ$. The set $\cJ$ equals $\cJ = \{\emptyset, \{1\}\}$ and $E_\emptyset = E$, $\Gamma_\emptyset = \bR$ and $E_{\{1\}} = \{0\}$. The tangent spaces are given by $T_{0} E = \Gamma_{\{1\}} = [0,\infty)$ and $T_xE = \bR$ for $x \in (0,1)$.
\end{example}

\begin{example}[Quarter space]
	Consider $E = [0,\infty)^2$. Then for $(x_1,v_1) = ((0,0),(1,0)), (x_2,v_2) = ((0,0),(0,1))$, we have $E = \cB_{x_1,v_1} \cap \cB_{x_2,v_2}$. The set $\cJ$ equals $\cJ = \{\emptyset, \{1\}, \{2\},\{1,2\}\}$. Furthermore:
	\begin{align*}
	E_\emptyset & = E, & \Gamma_{\emptyset}&  = \bR^2, \\
	E_{\{1\}} & = \left\{(x,y) \, \middle| \, x = 0, y \geq 0 \right\}, & \Gamma_{\{1\}} &  = \left\{(v_1,v_2) \, \middle| \, v_1 \geq 0\right\}, \\
	E_{\{2\}} & = \left\{(x,y) \, \middle| \, x \geq 0, y = 0 \right\}, & \Gamma_{\{2\}}&  = \left\{(v_1,v_2) \, \middle| \, v_2 \geq 0\right\}, \\
	E_{\{1,2\}} & = \left\{(x,y) \, \middle| \, x = 0, y = 0 \right\}, & \Gamma_{\{1,2\}}&  = \left\{(v_1,v_2) \, \middle| \, v_1,v_2 \geq 0\right\}. 
	\end{align*}
	Thus, the tangent spaces are given by 
		\begin{align*}
		T_{(0,0)} E & = J_{\{1,2\}} = [0,\infty)^2, \\
		T_{(x,0)} E & = J_{\{1\}} = [0,\infty) \times \bR & \text{if } x > 0, \\
		T_{(0,y)} E & = J_{\{2\}} = \bR \times [0,\infty) & \text{if } y > 0, \\		
		T_{(x,y)} E & = J_\emptyset = \bR^2 & \text{if } x,y > 0.
		\end{align*}
\end{example}

\begin{example}[Quarter space without corner]
	Consider $E = [0,\infty)^2 \setminus \{(0,0)\}$. Then for $(x_1,v_1) = ((0,0),(1,0)), (x_2,v_2) = ((0,0),(0,1))$, $(y_1,w_1) = ((0,0),(1,1))$, we have $E = \cB_{x_1,v_1} \cap \cB_{x_2,v_2} \cap \cB_{y_1,w_1}^\circ$. The set $\cJ$ equals $\cJ = \{\emptyset, \{1\}, \{2\}\}$. Furthermore:
	\begin{align*}
	E_\emptyset & = E, & \Gamma_{\emptyset}&  = \bR^2, \\
	E_{\{1\}} & = \left\{(x,y) \, \middle| \, x = 0, y < 0 \right\}, & \Gamma_{\{1\}} &  = \left\{(v_1,v_2) \, \middle| \, v_1 \geq 0\right\}, \\
	E_{\{2\}} & = \left\{(x,y) \, \middle| \, x > 0, y = 0 \right\}, & \Gamma_{\{2\}}&  = \left\{(v_1,v_2) \, \middle| \, v_2 \geq 0\right\}. 
	\end{align*}
	In this case, the tangent spaces are the same as in previous example, except for that the point $(0,0)$ is not included in the space $E$.
\end{example}

\subsection{A generating collection of Hamiltonians and an associated Lagrangian} \label{section:generating_sets_of_Hamiltonians}

Our goal is to prove a large deviation principle for the dynamics of stochastic systems, and write the rate function in Lagrangian form. The Lagrangian will as usual be given as the Legendre transform of a Hamiltonian. Due to the discontinuous nature of the dynamics at the boundary, this Hamiltonian will be constructed as the maximum over Hamiltonians that take into account all possible behaviours of the dynamics close to the boundary. We end this section with a simple example illustrating the construction.

\begin{definition} \label{definition:generating_Hamiltonians}
	Let $E$ be a $d$ dimensional convex Polyhedron. Suppose for each $J \in \cJ$, there is a map $\cH_J : E_J \times \bR^d \rightarrow \bR$  such that 
	\begin{enumerate}[(a)]
		\item for each $x \in E_J$ the map $p \mapsto \cH_J(x,p)$ is convex;
		\item the map $\cH_J$ is continuously differentiable in $(x,p)$; 
		\item the gradient $\partial_p \cH_J(x,p) \in \Gamma_J$;
	\end{enumerate}
	We say that $\bH := \{\cH_J\}_{J \in \cJ(E)}$ is a \textit{generating set of Hamiltonians}. For a generating set of Hamiltonians, denote
	\begin{equation*}
	\cH_\dagger(x,p)  := \max_{J \in \cJ: \, x \in E_J} \cH_J(x,p), \qquad \cH_\ddagger(x,p)  := \min_{J \in \cJ: \, x \in E_J} \cH_J(x,p).
	\end{equation*}
	We additionally write $H_\dagger \subseteq C_c^2(E) \times C_b(E)$ and $H_\ddagger \subseteq C_c^2(E) \times C_b(E)$ for the operators defined by $H_\dagger f(x) := \cH_\dagger(x,\nabla f(x))$ and $H_\ddagger f(x) := \cH_\ddagger(x,\nabla f(x))$.
\end{definition}

\begin{definition} \label{definition:construction_of_L}
	Let $\bH = \{\cH_J\}_{J \in \cJ(E)}$ be a generating set of Hamiltonians. 	
	\begin{enumerate}[(a)]
		\item For each $J \in \cJ(E)$, define $\cL_J : \bR^d \times \bR^d \rightarrow [0,\infty]$ and $\cL : E \times \bR^d \rightarrow [0,\infty]$ by 
		\begin{equation*}
		\cL_J(x,v) = \sup_{p \in \bR^d} \left\{pv - \cH_J(x,p) \right\}, \qquad \cL(x,v) = \sup_{p \in \bR^d} \left\{pv - \cH_\dagger(x,p) \right\}.
		\end{equation*}
		\item Denote by $\widehat{\cL} : E \times \bR^d \rightarrow [0,\infty]$ the map 
		\begin{equation*}
		\widehat{\cL}(x,v) := \inf \left\{ \sum_{J : x \in E_J} \lambda_J \cL_J(x,v_J) \, \middle| \, \sum_{J : x \in E_J} \lambda_J v_J = v, \, \sum_{J : x \in E_J} \lambda_J= 1, \lambda_J\geq 0 \right\}.
		\end{equation*}		
	\end{enumerate}
\end{definition}

We start with a collection of preliminary properties of $\cL, \widehat{\cL}$ and their Legendre duals. 

\begin{lemma} \label{lemma:properties_of_L}
	Let $\{\cH_J\}_{J \in \cJ(E)}$ be a generating set of Hamiltonians.
	\begin{enumerate}[(a)]
		\item For each $x \in E$ the map $v \mapsto \widehat{\cL}(x,v)$ is the convex hull of the Lagrangians $v \mapsto \cL_J(x,v)$, i.e. the largest convex function that lies below all $\cL_J$'s;
		\item For each $x \in E$, the map $v \mapsto \cL(x,v)$ is the lower semi-continuous regularization of $v\mapsto \widehat{\cL}(x,v)$.
	\end{enumerate}
\end{lemma}

\begin{proof}
	(a) follows from \cite[Theorem 5.6]{Ro70}. (b) follows from \cite[Theorem 16.5]{Ro70}.
\end{proof}

The lemma shows that our Lagrangian turns out to be similar to the one in \cite{DuIiSo90} in the sense that it it can be obtained as the lower semi-continuous hull of the Lagrangians that are obtained for the different types of boundary behaviour. Note that we do not re-define this convex hull to equal $\infty$ when the speed points outward. This issue will be taken care of by only considering trajectories that stay inside the state-space, i.e. in $\cA\cC(E)$.

\smallskip

We end this section by working out the Hamiltonians and Lagrangians for a simple example arising from the large deviations of a rescaled a birth-death process on $\bN$ with homogeneous jump rates $1$ for $k \mapsto k+1$ and $1$ for $k \mapsto k-1$ (if $k > 0$).

\begin{example}
	Let $E = [0,\infty)$, so that $E = B_{0,1}$. See Example \ref{example:half_line} for further notation. As $\cJ = \{\emptyset, \{1\}\}$ our generating set of Hamiltonians has two elements:
	\begin{align*}
	\cH_\emptyset(x,p) & =  \left[e^p -1\right] + \left[e^{-p} - 1\right], \\
	\cH_{\{1\}}(x,p) & =  \left[e^p -1\right].
	\end{align*}
	$\cH_\emptyset$ captures the behaviour of the process in the interior and consists of two terms corresponding to jumps up and down. $\cH_{\{1\}}$ captures the behaviour at the boundary point $\{0\}$ and therefore only consists of jumps going up. Note that indeed $\partial_p \cH_{\{1\}}(0,p) \subseteq [0,\infty)$.
	
	\smallskip
	
	We obtain $\cH_\dagger(x,p) = \cH_\ddagger(x,p) = \cH_\emptyset(x,p)$ for $x > 0$ and
	\begin{align*}
	\cH_\dagger(0,p) & = \begin{cases}
	\left[e^p -1\right] + \left[e^{-p} - 1\right] & \text{if } p \leq 0, \\
	\left[e^p -1\right] & \text{if } p \geq 0,
	\end{cases} \\
	\cH_\ddagger(0,p) & = \begin{cases}
	\left[e^p -1\right] & \text{if } p \leq 0, \\
	\left[e^p -1\right] +  \left[e^{-p} - 1\right] & \text{if } p \geq 0.
	\end{cases} 
	\end{align*}
	Performing Legendre transformation, we obtain
	\begin{align*}
	\cL_\emptyset(x,v) & = v \log \left(\frac{v + \sqrt{v^2 + 4}}{2} \right) - \sqrt{v^2 + 4} + 2, \\
	\cL_{\{1\}}(0,v) & = v \log v - v + 1.
	\end{align*}
	We thus find $\cL(x,v) = \cL_\emptyset(x,v)$ for $x > 0$ and
	\begin{equation*}
	\cL(0,v) = \begin{cases}
	\cL_\emptyset(0,v) & \text{if } v \leq 0, \\
	0 & \text{if } v \in [0,1], \\
	\cL_{\{1\}}(0,v) & \text{if } v \leq 1.
	\end{cases}
	\end{equation*}
	Clearly, the part $\cL(0,v)$ for $v < 0$ will not play a role in the final large deviation principle as we will exclude trajectories that leave $[0,\infty)$.
\end{example}

\section{Large deviations for a collection of examples} \label{section:LDP_for_examples}

Before diving into the general results in Sections \ref{section:framework_HJ_boundary} and \ref{section:abstract_LDP}, we give a variety of contexts in which these general results can be applied. The examples are not meant to give an exhaustive list, but rather illustrate that the general results are applicable in a variety of contexts. 

We begin with a collection of examples for the one-dimensional setting in Section \ref{section:one_d_examples} and proceed with large deviations for interacting species in Section \ref{section:multi_d_examples}.

\subsection{One dimensional examples} \label{section:one_d_examples}

We start with three examples in which there do not occur discontinuities in the Hamiltonian at the boundary,  i.e. $H_\dagger = H_\ddagger$:
\begin{enumerate}[(1)]
	\item Large deviations for birth and death processes with immigration;
	\item Large deviations for birth processes with jumps of arbitrary size.
	\item Large deviations for simple epidemic models.
\end{enumerate}
We state these results to illustrate the extent of applicability of our method. We will discuss in each of these three cases why the result complements the literature \cite{FK06,KrPa16,PaSa17}.

We  proceed with birth and death processes with immigration and harvesting. Harvesting introduces a discontinuity at the boundary: harvesting is not possible if there are no individuals.

As a final example, we discuss birth and death processes without any form of immigration or harvesting at the boundary.  In this setting our results do not apply, and we argue why indeed a result of the type given in all other examples is not expected to hold.

Finally, we mention that the one-dimensional results for the dynamic Curie-Weiss model in \cite{Kr16b} fall within the framework introduced in this paper and that the results can also be applied to the Wright-Fisher and Moran models with positive mutation rate or for one-dimensional models in the analysis of queuing sytems.

\subsubsection{Birth and death processes with immigration} \label{section:BD_immigration}

For $n\in \bN$, let $X_n$ be a Markov process on $\bN$ with jump rate $\lambda_n(q)$ corresponding to the birth of an individual, $\mu_n(q)$ for the death of one individual and $\rho_n(q)$ for the arrival of one individual through immigration if the population has size $q \in \bN$. The definition of $\lambda_n$ and $\mu_n$ implies $\lambda_n(0)=\mu_n(0)=0$. 

Thus, the process has generator :
\begin{equation*}
\cA_nf(q)=(\lambda_n(q)+\rho_n(q))\left(f(q+1)-f(q)\right)+\mu_n(q)\left(f(q-1)-f(q)\right)
\end{equation*}
To make sure that the martingale problem, see Corollary 8.3.2 in \cite{EK86}, is well-posed, we impose for each $n$ that 
\begin{equation*}
\sup_{q \in \bN} \frac{\lambda_n(q) + \rho_n(q) + \mu_n(q)}{1+q} < \infty. 
\end{equation*}

\begin{theorem} \label{theorem:birth_death_immigration}
	Let $E = [0,\infty)$. Suppose that there are positive and continuous function $\lambda, \mu$, and $\rho$, such that for every compact set $K \subseteq E$:
	\begin{equation}\label{scale_condition_b&d}
	\lim_{n\rightarrow \infty} \sup_{x\in K, xn \in \bN} \left|\frac{1}{n}\lambda_n(nx)-\lambda(x)\right| +  \left|\frac{1}{n}\mu_n(nx)-\mu(x)\right| + \left|\frac{1}{n}\rho_n(nx)-\rho(x)\right| = 0.
	\end{equation}
	Moreover suppose that :
	
	\begin{enumerate}[(a)]
		\item $\lambda(0) = \mu(0) = 0$,  $\rho(0)>0$ and $\lambda$ and $\mu$ are strictly positive on $(0,\infty)$;
		\item the birth rate $\lambda$ and the immigration rate $\rho$ satisfy $\lambda(x)+\rho(x)=\cO(x\log(x))$ at $\infty$;
		\item $\exists$ $\delta>0$ such that $\mu$ is increasing on $[0,\delta]$.
	\end{enumerate}
	
	Let $X_n$ be solutions to the martingale problem or $\cA_n$. Suppose that $\tfrac{1}{n}X_n(0)$ satisfies a large deviation principle on $\bR^+$ with speed $n$ and with good rate function $I_0$. Then the process $t \mapsto \tfrac{1}{n}X_n(t)$ satisfies a large deviation principle on $D_{\bR^+}(\bR^+)$ with speed $n$ and with good rate function $I$:
	\begin{equation*}
	I(\gamma)=
	\begin{cases}
	I_0(\gamma(0))+\int_0^\infty \cL(\gamma(s),\dot\gamma(s))ds &b\text{if }\gamma \in \cA \cC(\bR^+), \\
	\infty&\text{otherwise}.
	\end{cases}
	\end{equation*}
	The Lagrangian $\cL$ is the Legendre transform of 
	\begin{equation}\label{equation:hamiltonian_birth_death}
	\cH(x,p)=\left(\lambda(x)+\rho(x)\right)\left[e^p-1\right]+\mu(x)\left[e^{-p}-1\right].
	\end{equation}		
	
\end{theorem}

\begin{remark}
	The result is a variant (without the diffusion part) of Theorem 10.22 of \cite{FK06} for Lévy processes on $\bR^d$ adapted to the context of $[0,\infty)$. The result is a complement to \cite{KrPa16,PaSa17} due to the non-compactness of the state-space and possibly non-Lipschitz rates.
\end{remark}

\subsubsection{Growing populations} \label{section:LDP_for_growing_populations}

As in Section \ref{section:BD_immigration}, we consider an evolving population $X_n$. In this setting however, we assume that the population can only grow and starts with non-zero size. 

Each individual gives $k$ offspring at a rate that depends on the population size: the rate at which individuals give $k$ offspring depends on $n$: $v_{k,n}(q)$. The generator of this process, if it exists, is given by
\begin{equation*}
\cA_n f(q) = \sum_{k=1}^\infty q v_{k,n}(q) \left[f(q+k) - f(q)\right].
\end{equation*}
As the rates of this jump process are potentially unbounded, we impose for each $n$ the conditions
\begin{equation} \label{eqn:increasing_jump_process_well_posedness_conditions}
\sup_q \sum_{k = 1}^\infty q v_{k,n}(q) < \infty, \qquad \sup_q \sum_{k=1}^\infty k v_{k,n}(q) < \infty.
\end{equation}
These two bounds and  Corollary 8.3.2 in \cite{EK86} imply that the martingale problem for $\cA_n$ has a unique solution.

\begin{theorem} \label{theorem:LDP_growing_population}
	Let $E = (0,\infty)$. Suppose \eqref{eqn:increasing_jump_process_well_posedness_conditions} is satisfied. Suppose in addition that there are continuous functions $v_k$ on $E$ such that for each compact set $K \subseteq E$, we have
	\begin{equation} \label{eqn:LDP_growing_population_convergence_condition}
	\lim_{n \rightarrow \infty} \sup_{x \in K, xn \in \bN} \left|v_{k,n}(nx) - v_k(x) \right| = 0.
	\end{equation}
	
	In addition, the functions $v_k$ satisfy
	\begin{enumerate}[(a)]
		\item $\sum_{k =1}^\infty v_k(x) > 0$ for all $x \in (0,\infty)$;
		\item the function $x \mapsto \sum_{k=1}^\infty k v_k(x)$ is bounded;
		\item \label{item:growth_condition_growing_population} there is some $\alpha > 0$ such that
		\begin{equation*}
		\sup_x \sum_{k=1}^\infty x v_k(x) \frac{k^2}{(1+x)^2} e^{\alpha k/(1+x)} < \infty.
		\end{equation*}
	\end{enumerate}
	
	Let $X_n$ be  solutions to the martingale problem for $\cA_n$. Suppose that $\tfrac{1}{n}X_n(0)$ satisfies the large deviation principle on $E$ with speed $n$ and with good rate function $I_0$. Then the processes $t \mapsto \tfrac{1}{n}X_n(t)$ satisfy the large deviation principle on $D_E(\bR^+)$ with good rate function
	\begin{equation*}
	I(\gamma) = \begin{cases}
	I_0(\gamma(0)) + \int_0^\infty \cL(\gamma(s),\dot{\gamma}(s)) \dd s & \text{if } \gamma \in \cA \cC(0,\infty), \\
	\infty & \text{otherwise},
	\end{cases}
	\end{equation*}
	where $\cL(x,v) = \sup_p pv - \cH(x,p)$ and where	$H : E \times \bR \rightarrow \bR$ is defined as
	\begin{equation} \label{eqn:Hamiltonian_positve_jumps}
	\cH(x,p) = \sum_{k=1}^\infty x v_k(x) \left[ e^{kp} - 1 \right].
	\end{equation}
\end{theorem}

\begin{remark}
	As above the result is a variant (without the diffusion part) of Theorem 10.22 of \cite{FK06} for Lévy processes on $\bR^d$ adapted to the context of $(0,\infty)$. Note that due to the possibly infinite number of jump types, non-compact state-space and uni-directionality of the model the results of \cite{KrPa16,PaSa17} do not apply.
\end{remark}
	
\begin{remark}
	The growth condition in \ref{item:growth_condition_growing_population} arises when controlling the growth of the large deviations of the process. It is satisfied in any case in which $\nu$ has finite support uniformly in $x$. 
	
	In the case that $\nu$ has unbounded support, a similar expression without the exponent would appear when controlling the range of the process without looking at the large deviations. This is achieved by shifting the linear part into a drift term, and then controlling the second order Taylor expansion. This is reminiscent of the situation for the characteristics of a Lévy process.
	On the large deviation scale, we have to control the exponential function, which leads to the exponential factor.
\end{remark}

\begin{example}[Yule process] \label{example:Yule}
	The Yule process \cite{AtNe72} is defined by taking rates independent of $n$ and $v_{1}(q) = 1$ and $v_k(q) = 0$ for all $k > 1$. The Hamiltonian has the simple form $\cH(x,p) = x \left[e^p - 1\right]$. Conditions (a) and (b) of Theorem \ref{theorem:LDP_growing_population} are trivially satisfies, whereas for (c) every $\alpha > 0$ works.
\end{example}

\begin{example}[Poisson $\beta$ number of children]
	Consider the case in which each individual has an offspring that has Poisson size with parameter $\beta$. In this case the Hamiltonian has the form
	\begin{equation*}
	\cH(x,p) = \sum_{k=1}^\infty x \frac{\beta^k}{k!} e^{-\beta} \left[e^{kp} - 1\right].
	\end{equation*}
	Condition (a) of Theorem \ref{theorem:LDP_growing_population} is trivially satisfied, whereas (b) is satisfied as we are considering a Poisson random variable. 
	
	For Condition (c), choose $\alpha = \frac{1}{2}\beta$. Note that
	\begin{equation*}
	\sum_{k=1}^\infty x \frac{\beta^k}{k!} e^{- \beta} \frac{k^2}{(1+x)^2} e^{\frac{1}{2}\beta /(1+x) }  \leq \sum_{k=1}^\infty  \frac{k^2 \beta^k}{k!} e^{- \frac{1}{2}\beta} < \infty.
	\end{equation*}
	Thus also for this case the large deviation principle holds.
\end{example}

\subsubsection{Spread of a disease: a model with susceptible and infected individuals}

We consider a simple SI model (with Susceptible and  Infected individuals) for the spread of a disease in a population \cite{Al10}. The population has constant size $n$, and is divided in two groups: $X_n$ susceptible, and $Y_n$ infected, individuals. Because the total size of the population is fixed, we write $Y_n = n - X_n$, so that the whole system can be described in terms of $X_n$.	

The susceptible individuals contract the disease upon contact with an infected individual at rate $\frac{\beta}{n}$, $\beta > 0$, which leads to a generator:
\begin{equation*}
\cG_n f(q)= \frac{\beta}{n}q(n-q)\left(f(q+1)-f(q)\right), \qquad q\in \{0,\dots,n\}.
\end{equation*}
In general, we can replace the rate $\frac{\beta}{n}q(n-q)$ by a rate $c_n(q)$ that satisfies $c_n \geq 0$ and $c_n(0) = 0$:
\begin{equation*}
\cA_n f(q)= c_n(q)\left(f(q+1)-f(q)\right), \qquad q\in \{0,\dots,n\}.
\end{equation*}

\begin{theorem}\label{theorem:LDP_SI_model}
	Let $E=(0,1]$. Suppose that there is a non-negative continuous function $c \in C(0,1]$ such that for every compact set $K \subseteq E$:
	\begin{equation*}
	\lim_{n\rightarrow \infty} \sup_{x\in K, xn \in \bN} \left|\frac{1}{n}c_n(nx)-c(x)\right| = 0.
	\end{equation*}
	Moreover suppose that :
	\begin{enumerate}
		\item $c(x) > 0$ for $x \in (0,1)$ and $c(1) = 0$,
		\item $c$ is decreasing in a neighbourhood of $1$.
	\end{enumerate}
	Let $X_n$ be solutions to the martingale problem for $\cA_n$. Suppose that $\frac{1}{n}X_n(0)$ satisfies the large deviation principle on E with speed $n$ and with good rate function $I_0$. Then the processes $\frac{1}{n}X_n$ satisfy the large deviation principle on $D_E(\bR^+)$ with good rate function
	\begin{equation*}
	I(\gamma) = \begin{cases}
	I_0(\gamma(0)) + \int_0^\infty \cL(\gamma(s),\dot{\gamma}(s)) \dd s & \text{if } \gamma \in \cA \cC(0,1], \\
	\infty & \text{otherwise},
	\end{cases}
	\end{equation*}
	where $\cL(x,v) = \sup_p pv - \cH(x,p)$ and where	$H : E \times \bR \rightarrow \bR$ is defined as
	\begin{equation*}
	\cH(x,p)= c(x)\left[e^p-1\right]
	\end{equation*}
\end{theorem}

\begin{remark}
	$c$ does not need to be bounded in a neighbourhood of $0$.
\end{remark}

\begin{remark}
	As above, this result complements \cite{KrPa16} in the sense that the state-space is non-compact, the model uni-directional and the rates are possibly non-Lipschitz and non-bounded.
\end{remark}

\subsubsection{Birth and death processes with immigration and harvesting} \label{subsection:birth_death_immigration_harvesting}

We adjust the example in Section \ref{section:BD_immigration} by adding harvesting.

\smallskip

Again, for $n\in \bN$, let $X_n$ be a Markov process on $\bN$ with birth-rate $\lambda_n(q)$ , death-rate $\mu_n(q)$ for immigration-rate $\rho_n(q)$. In addition, a breeder kills a single individual at a rate $\beta_n(q)$ that depends possibly on the population size $q$. To make sure that the boundary $0$ can be accessed as well as exited sufficiently fast, we assume that there exists a $\varepsilon > 0$ such that 
	\begin{equation*}
	\inf_n \inf_{q \in \{0,\dots,\lfloor \varepsilon n\rfloor\}}\rho_n(q) \geq \varepsilon, \qquad \inf_n \inf_{q \in \{1,\dots,\lfloor \varepsilon n\rfloor\}}\beta_n(q) \geq \varepsilon.
	\end{equation*}
	As before, the generator reads
	\begin{equation*}
	\cA_nf(q) =
	\begin{cases}
	(\lambda_n(q)+ \rho_n(q))(f(q+1)-f(q))+(\mu_n(q) + \beta_n(q))(f(q-1)-f(q))&\text{if } q \in \bN\setminus\{0\},\\
	n\rho_n(0)(f(1)-f(0))&\text{if } q =0.\\
	\end{cases}
	\end{equation*} 		
As before, we also impose for each $n$ that 
\begin{equation*}
\sup_{q \in \bN} \frac{\lambda_n(q) + \rho_n(q) + \mu_n(q) + \beta_n(q)}{1+q} < \infty. 
\end{equation*}
to ensure well-posedness of the martingale problems, cf. Corollary 8.3.2 in \cite{EK86}.

\begin{theorem}\label{theorem:birth_death_immigration_harvesting}
	Let $E=[0,\infty)$. Suppose that there are positive and continuous function $\lambda, \mu, \rho$, and $\beta$, such that for every compact set $K \subseteq E$ :
		\begin{multline}\label{eqn:scaling_harvesting}
		\lim_{n\rightarrow \infty} \sup_{x\in K: nx \in \bN} \left\{ \left|\frac{1}{n}\lambda_n(nx)-\lambda(x)\right| +  \left|\frac{1}{n}\mu_n(nx)-\mu(x)\right| \right. \\
		\left. + \left|\frac{1}{n}\rho_n(nx)-\rho(x)\right| + \left|\frac{1}{n}\beta_n(nx)-\beta(x)\right| \right\} = 0.
		\end{multline}
	Moreover suppose that :
	\begin{enumerate}[(a)]
		\item $\lambda(0) = \mu(0) = 0$,  $\rho(0), \beta(0) >0$ and $\lambda$ and $\mu$ are strictly positive on $E^\circ$;
		\item the birth rate $\lambda$ and the immigration rate $\rho$ satisfy $\lambda(x)+\rho(x)=\cO(x\log(x))$ at $\infty$.
	\end{enumerate}
	Suppose that $(\tfrac{1}{n}X_n(0))$ satisfies a large deviation principle with speed $n$ and with good rate function $I_0$. Then the process $t \mapsto \tfrac{1}{n}X_n(t)$ satisfies a large deviation principle on $D_E(\bR^+)$ with speed $n$ and good rate function $I$:
	\begin{equation*}
	I(\gamma)=
	\begin{cases}
	I_0(\gamma(0))+ \int_0^\infty \cL(\gamma(s),\dot\gamma(s))ds & \text{if } \gamma \in \cA \cC(\bR^+), \\
	\infty & \text{otherwise}.
	\end{cases}
	\end{equation*}
	
	The Lagrangian is given in Definition \ref{definition:construction_of_L} constructed from generating set of Hamiltonians $\{H_{\partial-},H_{\emptyset}\}$, $(\partial_-=0)$ :
	\begin{align}
	\cH_\emptyset(x,p)&=(\lambda(x) +\rho(x))\left[e^p-1\right]+(\mu(x) + \beta(x))\left[e^{-p}-1\right], \label{equation:hamiltonian_interior_harvest} \\
	\cH_{\partial_-}(p)&=\rho(0)\left[e^p-1\right]. \label{equation:hamiltonian_zero_harvest}
	\end{align}
\end{theorem}

\begin{remark}
	In this example, we work in the context that the Hamiltonian depends on the location $x \in E$. In this context $E = [0,\infty) = \cB_{0,1}$. Thus $\cJ(E) = \{1\}$ as the left boundary point $0$ is included in $E = [0,\infty)$. For intuitive understanding, we write $\partial_-$ for the set $\{1\}$ as it corresponds as the set $E_{\{1\}}$ equals the left boundary point $\{0\}$.
	
	To fulfil Definition \ref{definition:generating_Hamiltonians}, we need to specify continuously differentiable $\cH_{\emptyset} : E \times \bR \rightarrow \bR$ and $\cH_{\{1\}} = \cH_{\partial_-} : \{0\} \times \bR \rightarrow \bR$ such that $\partial_p \cH_{\emptyset}(x,p) \subseteq \bR$ and $\partial_p \cH_{\partial_-}(0,p) \subseteq [0,\infty)$.
	
	This is indeed satisfied as can be read of from \eqref{equation:hamiltonian_interior_harvest} and \eqref{equation:hamiltonian_zero_harvest}. 
\end{remark}

\subsection{A discussion on the failure of the large deviation principle for pure birth and death processes}	\label{subsection:failure_of_LDP}

We end our discussion of one-dimensional processes by showing that a large deviation principle of the type considered above is not to be expected for birth and death processes without immigration or harvesting,  or more generally without a mechanism by which the boundary and the interior show similar behaviour. Our example is of similar nature to that of Example E in \cite{ShWe05}.

Consider the process $X_n(t)$ with generator
\begin{equation*}
\cA_nf(q)= a q \left(f(q+1)-f(q)\right) + b q \left(f(q-1)-f(q)\right)
\end{equation*}
where $a > 0, b \geq 0$. Note that $a = 1, b = 0$ corresponds to the Yule process of Example \ref{example:Yule}.

\smallskip

We argue first that the large deviation principle like the ones in the other examples cannot hold for the Yule process on $[0,\infty)$, unlike  on $(0,\infty)$ as in Section \ref{section:LDP_for_growing_populations}. If this result would be true then our theory would indicate that the Hamiltonian $\cH$ corresponding to the principle has the form
\begin{equation*}
\cH(x,p) = x \left[e^p - 1\right],
\end{equation*}
so that the path-space Lagrangian becomes
\begin{equation*}
\cL(x,v) = v \log \frac{v}{x} - v + x.
\end{equation*}
The trajectory $t \mapsto t^2$ can be seen to have finite cost. This indicates that the large deviation cost to go from $0$ to $1$ in time $1$ is finite. However, if $X_n(0) = 0$ for all $n$, the large deviation cost to be in $1$ at time $1$ is infinity.

\smallskip

In the general setting, with 
\begin{equation*}
\cH(x,p) = ax \left[e^p - 1\right] + b x \left[ e^{-p} - 1\right],
\end{equation*}
$a > 0, b \geq 0$, one can show that the large deviation contribution of having no jumps of size $-1$ is finite. As the cost to go from $0$ to $1$ in finite time, just using positive jumps is finite as for the Yule process, we find that also in this case a contradiction if $X_n(0) = 0$ for all $n$.

\smallskip

This show that care is needed when formulating a large deviation principle for the rescaled population size of a pure birth-death process when allowing $0$ as a starting point. Even though we have not talked about the technical ingredients for our proofs, we want to indicate that this implies that Condition \ref{condition:boundary} and the corresponding Theorem \ref{theorem:comparison_principle_1d} for one-dimensional processes leaves out the pure birth-death processes for good reason.

\subsection{Large deviations for interacting species with immigration and harvesting} \label{section:multi_d_examples}

We proceed our discussion with a general multi-dimensional model. In this setting, we only offer one set  of results that we think covers a range of interesting models. More models can be treated on a model by model basis by using the methods of Sections \ref{section:framework_HJ_boundary} and \ref{section:abstract_LDP}.

For each $n$, we consider $k$ different interacting species that evolve in time. We denote their population sizes at time $t$ by $(X^1(t),\dots,X^k(t))$.

We will denote possible transitions of the Markov process by vectors $\Gamma = (\Gamma_1,\dots,\Gamma_l) \in \bZ^k$ indicating how many individuals of each type get born of die. We write $\delta_i \in \bZ^k$ for the vector with a one in location $i$ and zero elsewhere. We consider three kinds of transitions, denoted by sets $T_{int},T_{imm},T_{har}$ corresponding to interaction between or within species, immigration and harvesting.

\begin{condition} \label{condition:interacting_species}
	We assume that the cardinality of $T_{int}$ is finite. In addition, we assume
	\begin{description}
		\item[interaction] For each $\Gamma \in T_{int}$, we have a rate $r_{\Gamma,n}(q)$ at which the transition from $q$ to $q + \Gamma$ occurs. Clearly, if there are no individuals of species $i$, then the rate of a transition in which individuals of species $i$ die is $0$. In other words, if there is an $i$ with $\Gamma_i <0$ and $q_i = 0$, then $r_{\Gamma,n}(q) = 0$. 
		\item[immigration] $T_{imm} = \left\{\delta_i \, \middle| \, i \in \{1,\dots,k\}\right\}$. The corresponding transition rate $a_{i,n}(q)$ is assumed to be strictly positive. 
		\item[harvesting] $T_{har} = \left\{- \delta_i \, \middle| \, i \in \{1,\dots,k\}\right\}$. The corresponding transition rate $b_{i,n}(q)$ is assumed to be strictly positive. 
	\end{description}
\end{condition}

The rates lead to a generator $A_n$ that reads
\begin{multline} \label{eqn:generator_interacting_species}
A_nf(q) = \sum_{\Gamma \in T_{int}} r_{\Gamma,n}(q) \left[f(q + \Gamma) - f(q) \right] + \sum_{i} a_{i,n}(q) \left[f(q + \delta_i) - f(q) \right] \\
+ \sum_{i} \bONE_{\{q_i\neq 0\}}b_{i,n}(q) \left[f(q - \delta_i) - f(q) \right]. 
\end{multline}
For well-posedness of the martingale problems, we impose for each $n$ that
\begin{equation} \label{eqn:well_poseness_multi_d}
\sup_{q \in \bN^k } \frac{\sum_{\Gamma \in T_{int}} r_{\Gamma,n}(q)+  \sum_{i} a_{i,n}(q)
	+ \sum_{i}b_{i,n}(q)}{1 + \sum_i q_i} < \infty.
\end{equation}
Well posedness is established by Theorem 8.3.1 in \cite{EK86} by taking $\gamma(q) = \eta(q) = 1 + \sum_i q_i$.

\smallskip

We will consider the large deviations on the space $E := [0,\infty)^k$ after rescaling the process by $n^{-1}$. For a momentum $p \in \bR^d$, denote by $\Gamma \cdot p$ the vector obtained by component-wise multiplication $(\Gamma_1 p_1,\dots,\Gamma_k p_k)$. Note that $\delta_i \cdot p = p_i\delta_i$.

\begin{theorem} \label{theorem:interacting_species}
	Let  $E := [0,\infty)^k$ and let Condition \ref{condition:interacting_species} be satisfied.	Suppose that there are positive and continuous function $r_\Gamma,a_i,b_i$ with $\Gamma \in T_{int}$ and $i \in \{1,\dots,k\}$ so that for each compact set $K \subseteq E$:
	
	\begin{multline}\label{eqn:scaling_multi_d}
	\lim_{n\rightarrow \infty} \sup_{x\in K} \left\{ \sum_{\Gamma \in T_{int}} \left|\frac{1}{n}r_{\Gamma,n}(nx)-r_\Gamma(x)\right| +  \sum_{i} \left|\frac{1}{n}a_{i,n}(nx)-a_i(x)\right| \right. \\
	\left. + \sum_{i} \left|\frac{1}{n}b_{i,n}(nx)-b_i(x)\right| \right\} = 0.
	\end{multline}
	Moreover suppose that $a_i, b_i$, $i \in \{1,\dots,k\}$,  are strictly positive and $r_\gamma$, $\Gamma \in T_{int}$,  is non-negative.
	
	Finally assume that for any $\Gamma \in T_{int}$ with $\sum_i \Gamma_i > 0$ we have $r_\Gamma(x) = \cO(s(x)\log(s(x)))$ with $s(x) = \sum x_i$  at infinity. In addition, assume $a_i(x) + b_i(x) = \cO(s(x) \log s(x))$ at infinity.

	\smallskip
	
	Let $(X_n^1,\dots,X_n^k)$ be a solution to the martingale problem for $A_n$ as in \eqref{eqn:generator_interacting_species}. Suppose that $(\tfrac{1}{n}X_n^1(0), \dots, \tfrac{1}{n}X_n^k(0))$ satisfies a large deviation principle with speed $n$ and with good rate function $I_0$.

	Then the process $t \mapsto \left(\tfrac{1}{n}X_n^1(t),\dots,X_n^k(t)\right)$ satisfies a large deviation principle on $D_E(\bR^+)$ with speed $n$ and good rate function $I$:
	\begin{equation*}
	I(\gamma)=
	\begin{cases}
	I_0(\gamma(0))+ \int_0^\infty \cL(\gamma(s),\dot\gamma(s))ds & \text{if } \gamma \in \cA \cC(\bR^+,\bR^+), \\
	\infty & \text{otherwise}.
	\end{cases}
	\end{equation*}
	The Lagrangian is given in Definition \ref{definition:construction_of_L} constructed from generating set of Hamiltonians $\bH$ defined below. Write
	\begin{equation*}
	\cH_0(x,p) = \sum_{\Gamma \in T_{int}} r_\Gamma(x) \left[e^{\Gamma \cdot p} - 1 \right] 
	\end{equation*}
	and define $\bH$ as
		\begin{align*}
		\cH_\emptyset(x,p) & = \cH_0(x,p)+\sum_{i=1}^d a_i(x)(e^{p_i}-1)+b_i(x)(e^{-p_i}-1), \\
		\intertext{and for $J\in\cJ(E)$ and $x\in E_J$}
		\cH_J(x,p) & =  \cH_0(x,p)+\sum_{i\in J} a_i(x)(e^{p_i}-1) + \sum_{i\notin J}a_i(x)(e^{p_i}-1)+b_i(x)(e^{-p_i}-1).
		\end{align*}
\end{theorem}

\begin{remark}
	As before, we specify the boundaries. As $E = [0,\infty)^k = \bigcap_{i =1}^k \cB_{0,e_i}$, where $e_i$ is the unit-vector in the $i$-th direction, the set $\cJ(E)$ equals the set of all subsets of $\{1,\dots,k\}$. For $J \subseteq \{1,\dots,k\}$ we have $E_J = \left\{x \in [0,\infty)^k \, | \, x_i =0 \text{ if } i \in J \right\}$ and $\Gamma_J := \left\{p \in \bR^k \, | \, p_i \geq 0 \text{ if } i \in J \right\}$.

\end{remark}

To conclude this section, we describe a prey-predator system and a SI model. To apply the main result of this section, we assume in both cases that there is immigration and harvesting. 
Because interactions between different species bring rates that are more than linear in the coordinates (e.g two individuals of different species meeting will induce a quadratic rate in the coordinates), the martingale problem might not be well-posed (see \ref{eqn:well_poseness_multi_d}). Therefore we introduce a "carrying capacity" $\kappa$ that describe the maximal population of a given species the environment can sustain \cite{ER11}. Once that population has grown to this carrying capacity, competition for resources will prevent reproduction. If that population grows above $\kappa$ through immigration, the other species will still interact as if there was only $\kappa$ individuals, the environment being saturated.

\begin{example}[Prey and predator system]
	We consider two populations of individuals whose sizes are given by the vector $\left(X^1_n(t),X^2_n(t)\right)$. The individuals of the second species hunt that of the first species. On top of some harvesting and immigration rates, the different interaction rates given the current population sizes $\left(X^1_n(t),X^2_n(t)\right)=(q_1,q_2)$ are given by: 
	\begin{description}
		\item[Birth of a prey:] The prey population grows with reproduction but is limited by competition. $\Gamma_1=\left(1,0\right)$ with rate $r_{1,n}(q)=q_1(n\kappa-q_1)$ if $q_1\leq n\kappa$ and $r_{1,n}(q)=0$ otherwise.	
		\item[Death of a prey:] A prey dies upon meeting a predator at rate $\alpha$. $\Gamma_1=\left(-1,0\right)$ with rate $r_{2,n}(q)=\frac{\alpha}{n}q_1 q_2$ if $q_1\leq n\kappa$ and $r_{2,n}(q)=\alpha\kappa q_2$ otherwise.
		\item[Birth of a predator:] The predator population grows when a predator kills and eats prey at rate $\beta$. $\Gamma_3=\left(0,1)\right)$ with rate $r_{3,n}(q)=\frac{\beta}{n}q_1q_2$ if $q_1\leq n\kappa$ and $r_{3,n}(q)=\beta \kappa q_2$ otherwise. The cut-off corresponds to the maximum carrying capacity of the population of prey. That is: the population of prey is large enough so that the amount of food for the population of predators is essentially unlimited, implying that their reproduction rate only depends on their own population size.
		\item[Death of a predator:] The predators live for an exponential time with mean $\mu$. That is: $\Gamma_4=\left(0,-1)\right)$ and $r_{4,n}(q)=\mu q_2$.
	\end{description}	
	The model satisfies the conditions of Theorem \ref{theorem:interacting_species} with $\cH_0$ :
	\begin{multline*}
	\cH_0(x,p) = x_1(\kappa-(x_1\wedge \kappa)\left(e^{p_1}-1\right)+\alpha (x_1\wedge \kappa)x_2\left(e^{-p_1}-1\right)\\
	+ \beta (x_1\wedge \kappa) x_2 \left(e^{p_2}-1\right)+\mu y \left(e^{-p_2}-1\right).
	\end{multline*}
	
\end{example}	

\begin{example}[SI model with recovery and population dynamics]
	Consider a population where individuals are divided into three groups:  the susceptible, the infected, and the recovered immune individuals. We model the sizes of these three groups by the process $(X_n^1,X_n^2,X_n^3)$.Denote $S(x) = x_1 + x_2 + x_3$. Consider the following interaction transitions for the process given  $\left(X^1_n(t),X^2_n(t),X^3_n(t)\right)=(q_1,q_2,q_3)$: 
	\begin{description}
		\item[Infection:] A susceptible individual has a chance to fall sick upon meeting an infected individual . The transition vector reads $\Gamma_{1}=(-1,1,0)$ and the transition rate is $r_{1,n}(q)=\frac{\beta}{n}(q_1\wedge \kappa)q_2$. 
		\item[Recovery:] An infected individual recovers from the disease and gains immunity.  $\Gamma_{2}=(0,-1,1)$ and the transition rate is $r_{2,n}(q)=\alpha q_2$.
		\item[Birth of an individual:] Suppose that individual are born susceptible, $\Gamma_{3}=(1,0,0)$ and $r_{3,n}(q)=\frac{S(q)}{n}(n\kappa-S(q))$ if $s(q)\leq n\kappa$ and $r_{3,n}(q)=0$ otherwise.
		\item[Death of an individual:] for $k \in \{1,2,3\}$, we have transitions $\Gamma_{4,k}=-\delta_k$ with $r_{4,k,n}(q)=\mu_k q_k$.
	\end{description}

	The conditions for Theorem \ref{theorem:interacting_species} hold with $\cH_0$ :
	\begin{multline*}
	\cH_0(x,p) = \beta (x_1\wedge \kappa) x_2\left(e^{p_2-p_1}-1\right)+\alpha x_2 \left(e^{p_3-p_2}-1\right)\\
	+S(x)(\kappa-(S(x)\wedge \kappa))\left(e^{p_1}-1\right)+\sum_{k=1}^3 \mu_k x_k \left(e^{-p_k}-1\right).
	\end{multline*}
	
\end{example}

\subsection{Discussion on extendability of the methods} \label{section:extendability}

In the beginning of Section \ref{section:LDP_for_examples} we stated that our results serve as an illustration of the scope of the general framework in the paper. We discuss shortly the extent of applicability of our methods.

\smallskip

For the one dimensional setting, one can consider models that combine features of the examples of Section \ref{section:comparison_1d}. A quick scan of the proofs of these results shows that the three main conditions that need to be verified (notation and definitions will be introduced below) are:
\begin{itemize}
	\item There needs to be a good containment function $\Upsilon$ for $\bH$ (See Definition \ref{definition:good_containment_function} below). A containment function is a type of Lyapunov function that implies that the process remains in compact sets on finite time intervals with large probability.
	\item We need to verify Conditions \ref{condition:basic} and \ref{condition:boundary} for $\bH$. These conditions can simply be read of from the generating set of Hamiltonians.
\end{itemize}
It should be noted that Condition \ref{condition:basic} is not satisfied for the generating set of Hamiltonians that arises from the scaling limit of diffusions with reflecting boundary, see e.g. Sections 9.3 and 10.5 in \cite{FK06}. A relaxation of this condition that covers both this case, as well as reflected diffusions, would be of interest.

\smallskip

For the multi-dimensional setting, we have only considered the case $E = [0,\infty)^d$ in Theorem \ref{theorem:interacting_species} which in turn is immediately derived from Theorem \ref{theorem:comparison_principle_multid}. 

Theorem \ref{theorem:comparison_principle_multid} can be extended in a straightforward way by adapting Step 2 of its proof. The extension is restricted, however, to the context where the transitions that are discontinuous at the boundary are perpendicular to the boundary. 

The only work known to the authors that includes a comparison principle for Hamiltonians with discontinuous rates for non-perpendicular jumps is \cite{DuIiSo90}. In this setting, the rates, however, are spatially homogeneous. A comparison principle for a Hamiltonian that features a combination of inhomogeneous rates and discontinuity for non-perpendicular jumps would be of interest as well.
What is necessary for such a substantial extension is unclear to the authors.

\section{A general framework for Hamilton-Jacobi equations with a boundary} \label{section:framework_HJ_boundary}

In the approach to large deviations by \cite{FK06} and \cite{DuIiSo90}, also applied more recently in e.g. \cite{DFL11, CoKr17,KrReVe19}, the distributional information of the process at finite $n$ related to the large deviations is encoded in the solutions $f_n$ of a class of Hamilton-Jacobi equation $f - \lambda H_n f = h$, $\lambda > 0, h \in C_b(E_n)$. Given that the Hamiltonians $H_n$ have a natural limiting upper bound $H_\dagger$ and lower bound $H_\ddagger$, semi-relaxed limits $\overline{f}$ and $\underline{f}$ of $f_n$, see \eqref{eqn:semi_relaxed_limits} below, give a sub-solution and super-solution to
\begin{equation*}
f - \lambda H_\dagger f = h, \qquad f - \lambda H_\ddagger f = h
\end{equation*}
respectively. To show that the sub- and super-solution coincide and give a proper `solution' to this combination of equations, we need the comparison principle. This we carry out in the sections below.

Throughout the section, we assume that $E$ is a convex polyhedron with non-empty interior.

\subsection{Viscosity solutions to Hamilton-Jacobi equations}

\begin{definition}[Viscosity solutions]
	Let $H_\dagger \subseteq C_l(E) \times \cM_u(E,\overline{\bR})$ and $H_\ddagger \subseteq C_u(E) \times \cM_l(E,\overline{\bR})$ and let $\lambda > 0$ and $h \in C_b(E)$. Consider the Hamilton-Jacobi equations
	\begin{align}
	f - \lambda H_\dagger f & = h, \label{eqn:differential_equation_dagger} \\
	f - \lambda H_\ddagger f & = h. \label{eqn:differential_equation_ddagger}
	\end{align}
	We say that $u$ is a \textit{(viscosity) subsolution} of equation \eqref{eqn:differential_equation_dagger} if $u$ is bounded, upper semi-continuous and if, for every $f \in \cD(H_\dagger)$ there exists a sequence $x_n \in \mathbb{R}$ such that
	\begin{gather*}
	\lim_{n \uparrow \infty} u(x_n) - f(x_n)  = \sup_x u(x) - f(x), \\
	\lim_{n \uparrow \infty} u(x_n) - \lambda H_\dagger f(x_n) - h(x_n) \leq 0.
	\end{gather*}
	We say that $v$ is a \textit{(viscosity) supersolution} of equation \eqref{eqn:differential_equation_ddagger} if $v$ is bounded, lower semi-continuous and if, for every $f \in \cD(H_\ddagger)$ there exists a sequence $x_n \in \mathbb{R}$ such that
	\begin{gather*}
	\lim_{n \uparrow \infty} v(x_n) - f(x_n)  = \inf_x v(x) - f(x), \\
	\lim_{n \uparrow \infty} v(x_n) - \lambda H_\ddagger f(x_n) - h(x_n) \geq 0.
	\end{gather*}
	We say that $u$ is a \textit{(viscosity) solution} of equations \eqref{eqn:differential_equation_dagger} and \eqref{eqn:differential_equation_ddagger} if it is both a subsolution to \eqref{eqn:differential_equation_dagger} and a supersolution to \eqref{eqn:differential_equation_ddagger}.
	
	We say that \eqref{eqn:differential_equation_dagger} and \eqref{eqn:differential_equation_ddagger} satisfy the \textit{comparison principle} if for every subsolution $u$ to \eqref{eqn:differential_equation_dagger} and supersolution $v$ to \eqref{eqn:differential_equation_ddagger}, we have $u \leq v$.
\end{definition}

We will study the Hamilton-Jacobi equations for the operators $H_\dagger,H_\ddagger$ that have been constructed from a generating set of Hamiltonians $\{\cH_J\}_{J \in \cJ(E)}$ as in Definition \ref{definition:generating_Hamiltonians}.

\subsection{A general method to verify the comparison principle}

In this section, we give the main technical results used in the text above that can be used to verify comparison principles. These methods are fairly standard, do not use any structure of our particular setting in a crucial way, and follow those in \cite{CIL92,FK06,Kr16b,CoKr17}. The proofs of Proposition \ref{proposition:comparison_conditions_on_H} and Lemma \ref{lemma:control_on_H} do require minor adjustments of these methods and have therefore been given in Appendix \ref{appendix:abstract_comparison_verification} below. The results are based on a good containment function and a good penalization function.

Good containment functions play the role of a Lyapunov function and allow our analysis to be restricted to compact regions in $E$. The penalization functions are used in a distance like way.

\begin{definition} \label{definition:good_containment_function}
	Let $\{\cH_J\}_{J \in \cJ(E)}$ be a generating set of Hamiltonians.  We say that $\Upsilon : E \rightarrow \bR$ is a \textit{good containment function} (for $\bH = \{\cH_J\}_{J \in \cJ(E)}$) if
	\begin{enumerate}[($\Upsilon$a)]
		\item $\Upsilon \geq 0$ and there exists a point $x_0$ such that $\Upsilon(x_0) = 0$,
		\item $\Upsilon$ is twice continuously differentiable, 
		\item $\Upsilon$ has compact sub-level sets: for every $c \geq 0$, the set $\{x \in E \, | \, \Upsilon(x) \leq c\}$ is compact,
		\item we have $\max_{J \in \cJ(E)} \sup_{z \in E_J} \cH_J(z,\nabla \Upsilon(z)) < \infty$.
	\end{enumerate}
\end{definition}

\begin{definition}
	We say that $\Psi : E^2 \rightarrow \bR^+$ is a \textit{good penalization functions} if
	\begin{enumerate}[($\Psi$a)]
		\item $\Psi(x,y) = 0$ if and only if $x = y$;
		\item $\Psi$ is twice continuously differentiable in both coordinates;
		\item $(\nabla \Psi(\cdot,y))(x) = - (\nabla \Psi(x,\cdot))(y)$.
	\end{enumerate}
\end{definition}

The first result can be found as Proposition 3.7 of \cite{CIL92} or Lemma 9.2 in \cite{FK06}.

\begin{lemma}\label{lemma:doubling_lemma}
	Let $u$ be bounded and upper semi-continuous, let $v$ be bounded and lower semi-continuous, let $\Psi : E^2 \rightarrow \bR^+$ be a good penalization function and let $\Upsilon : E \rightarrow \bR^+$ have compact sub-level sets. Fix $\varepsilon > 0$. For every $\alpha >0$ there exist points $x_{\alpha,\varepsilon},y_{\alpha,\varepsilon} \in E$, such that
	\begin{multline*}
	\frac{u(x_{\alpha,\varepsilon})}{1-\varepsilon} - \frac{v(y_{\alpha,\varepsilon})}{1+\varepsilon} - \alpha \Psi(x_{\alpha,\varepsilon},y_{\alpha,\varepsilon}) - \frac{\varepsilon}{1-\varepsilon}\Upsilon(x_{\alpha,\varepsilon}) -\frac{\varepsilon}{1+\varepsilon}\Upsilon(y_{\alpha,\varepsilon}) \\
	= \sup_{x,y \in E} \left\{\frac{u(x)}{1-\varepsilon} - \frac{v(y)}{1+\varepsilon} -  \alpha \Psi(x,y)  - \frac{\varepsilon}{1-\varepsilon}\Upsilon(x) - \frac{\varepsilon}{1+\varepsilon}\Upsilon(y)\right\}.
	\end{multline*}
	Let $\{x_{\alpha,\varepsilon},y_{\alpha,\varepsilon}\}_{\alpha,\varepsilon}$ be such a collection of points. For every $\varepsilon > 0$ we have 
	\begin{enumerate}[(a)]
		\item The set $\{x_{\alpha,\varepsilon}, y_{\alpha,\varepsilon} \, | \,  \alpha > 0\}$ is relatively compact in $E$.
		\item All limit points of $\{(x_{\alpha,\varepsilon},y_{\alpha,\varepsilon})\}_{\alpha > 0}$ are of the form $(z,z)$ and for these limit points we have $u(z) - v(z) = \sup_{x \in E} \left\{u(x) - v(x) \right\}$.
		\item We have 
		\[
		\lim_{\alpha \rightarrow \infty}  \alpha \Psi(x_{\alpha,\varepsilon},y_{\alpha,\varepsilon}) = 0.
		\]
	\end{enumerate}
\end{lemma}

The following result gives us the explicit condition that can be used to verify the comparison principle.

\begin{proposition} \label{proposition:comparison_conditions_on_H}
	Let $\{\cH_J\}_{J \in \cJ(E)}$ be a generating set of Hamiltonians. Let $\Upsilon$ be a good containment function for $\{\cH_J\}_{J \in \cJ(E)}$ and let $\Psi$ be a good penalization function. Fix $\lambda >0$, $h \in C_b(E)$ and consider $u$ a subsolution to $f - \lambda H_\dagger f = h$ and  $v$ a supersolution to $f - \lambda H_\ddagger f = h$.
	
	For every $\alpha,\varepsilon >0$ let $x_{\alpha,\varepsilon},y_{\alpha,\varepsilon} \in E$ be such that
	\begin{multline} \label{eqn:comparison_principle_proof_choice_of_sequences}
	\frac{u(x_{\alpha,\varepsilon})}{1-\varepsilon} - \frac{v(y_{\alpha,\varepsilon})}{1+\varepsilon} -  \Psi_\alpha(x_{\alpha,\varepsilon},y_{\alpha,\varepsilon}) - \frac{\varepsilon}{1-\varepsilon}\Upsilon(x_{\alpha,\varepsilon}) -\frac{\varepsilon}{1+\varepsilon}\Upsilon(y_{\alpha,\varepsilon}) \\
	= \sup_{x,y \in E} \left\{\frac{u(x)}{1-\varepsilon} - \frac{v(y)}{1+\varepsilon} - \Psi_\alpha(x,y)  - \frac{\varepsilon}{1-\varepsilon}\Upsilon(x) - \frac{\varepsilon}{1+\varepsilon}\Upsilon(y)\right\}.
	\end{multline}
	Suppose that 
	\begin{multline}\label{condH:negative:liminf}
	\liminf_{\varepsilon \rightarrow 0} \liminf_{\alpha \rightarrow \infty} \cH_\dagger\left(x_{\alpha,\varepsilon},\nabla \Psi_\alpha(\cdot,y_{\alpha,\varepsilon})(x_{\alpha,\varepsilon})\right) \\
	- \cH_\ddagger\left(y_{\alpha,\varepsilon},\nabla \Psi_\alpha(\cdot,y_{\alpha,\varepsilon})(x_{\alpha,\varepsilon})\right) \leq 0,
	\end{multline}
	then $u \leq v$. In other words: the comparison principle holds for subsolutions to $f_\lambda H_\dagger f = h$ and supersolutions to $f - \lambda H_\ddagger f = h$.
\end{proposition}

We will establish \eqref{condH:negative:liminf} in Sections \ref{section:comparison_1d} and \ref{section:comparison_multi_d} below for a collection of Hamiltonians that appear in our applications.

\smallskip

The next lemma aids the verification of \eqref{condH:negative:liminf} by giving control on the sequences $(x_{\alpha,\varepsilon},y_{\alpha,\varepsilon})$. The result is an adaptation of Lemma 9.3 in \cite{FK06}. For a slightly less involved variant, see Lemma 5 in \cite{Kr16b}.

\begin{lemma} \label{lemma:control_on_H}
	Let $\{\cH_J\}_{J \in \cJ(E)}$ be a generating set of Hamiltonians. Let $h \in C_b(E)$ and $\lambda > 0$ and let $v$ be a supersolution to $f - \lambda_\ddagger H = h$. Let $\Psi$ be a good penalization function and $\Upsilon$ be a good containment function for $\{\cH_J\}_{J \in \cJ(E)}$. Moreover, for every $\alpha,\varepsilon >0$ let $x_{\alpha,\varepsilon},y_{\alpha,\varepsilon} \in E$ be as in \eqref{eqn:comparison_principle_proof_choice_of_sequences}. Then we have that
	\begin{equation} 
	\sup_{\varepsilon, \alpha} \cH_\ddagger\left(y_{\alpha,\varepsilon},\alpha (\nabla \Psi(\cdot,y_{\alpha,\varepsilon}))(x_{\alpha,\varepsilon})\right) < \infty. \label{eqn:control_on_H_sup} 
	\end{equation}
\end{lemma}

\subsection{The comparison principle for one dimensional systems} \label{section:comparison_1d}

We proceed by applying the general result for the verification of the comparison principle in the one-dimensional setting. For the verification, we need additional assumptions on the generating Hamiltonians. We start by interpreting Definitions \ref{definition:base_space}, \ref{definition:tangent_cones} and \ref{definition:generating_Hamiltonians} in this one-dimensional setting.

\smallskip

Consider $E \subseteq \bR$ be an interval and let $\partial_-, \partial_+$ be the boundaries of $E$ in the set $\bR\cup \{- \infty, \infty\}$. We will consider the setting that $E$ is open, half open or closed. Clearly, $E$ is of the form as in Definitions \ref{definition:base_space} and \ref{definition:tangent_cones}. We will use $-,+$ as indices instead of $1,\dots,k$ for the sets that generate $E$. 

In particular, if $\partial_- \in E$, then $E_- := \{\partial_-\}$ and $\Gamma_{\{-\}} = [0,\infty)$ and if $\partial_+ \in E$, then $E_+ := \{\partial_+\}$ and $\Gamma_{\{+\}} = (- \infty,0]$.

In this setting, a collection of generating Hamiltonians takes a particularly simple form, there is a base Hamiltonian $\cH_\emptyset$, as well as an Hamiltonian for each boundary point $z \in E \cap \partial E$. I.e. we have a generating set of Hamiltonians $\{\cH_\emptyset\} \cup \{\cH_z\}_{z \in \{-,+\}}$. Finally, condition (c) of Definition \ref{definition:generating_Hamiltonians} translates into
\begin{align*}
\partial_- & \in E \cap \partial E &  \Longrightarrow &&  &  \partial_p \cH_-(\partial_-,p) \geq 0, \\
\partial_+ & \in E \cap \partial E &  \Longrightarrow && & \partial_p \cH_+(\partial_+,p) \leq 0.
\end{align*}

To prove the comparison principle, we will use the following conditions on the Hamiltonian.

\begin{condition} \label{condition:basic}
	Let $\{\cH_\emptyset\} \cup \{\cH_z\}_{z \in \{-,+\}}$ be a generating set of Hamiltonians.
	\begin{description}
		\item[Behaviour in the interior] 
		For each direction $d \in \{-,+\}$ and all compact $K \subseteq E^\circ$ we have either (1) or (2)
		\begin{enumerate}[(1)]
			\item
			\begin{equation*}
			\lim_{p \rightarrow d \infty} \inf_{x \in K}  \cH_\emptyset(x,p) = \infty;
			\end{equation*}
			\item there is a continuous function $h^d : K \rightarrow \bR$ such that
			\begin{equation*}
			\lim_{p \rightarrow d\infty} \sup_{x \in K} \left|\cH_\emptyset(x,p) - h^d(x)\right| = 0.
			\end{equation*}
		\end{enumerate}
		\item[Relating the boundary to the interior]
		For each boundary $\partial_s \in E \cap \partial E$ it holds
		\begin{align*}
		-sp & \geq 0 & \Longrightarrow && \cH_{s}(\partial_s,p) & \geq  \cH_\emptyset(\partial_s,p); \\
		-sp & \leq 0 & \Longrightarrow && \cH_{s}(\partial_s,p) & \leq  \cH_\emptyset(\partial_s,p). 
		\end{align*}
	\end{description}
\end{condition}

The main step when verifying the comparison principle for points at the boundary is bounding the Hamiltonians $\cH_\dagger$ and $\cH_\ddagger$ by $\cH_\emptyset$. This step allows us to proceed as if we were in the interior. We thus pose additional conditions for the behaviour of $H_\emptyset$ close to the boundary. We distinguish between a weak and a strong condition. The weak condition is sufficient for the verification of the comparison principle, whereas the strong condition plays an important role when we take sums of Hamiltonians, see Lemma \ref{lemma:additive_structure_conditions_one_d_hamiltonians} below.

\begin{condition}[Conditions for the boundary] \label{condition:boundary}
	Let $\{\cH_\emptyset\} \cup \{\cH_z\}_{z \in \{-,+\}}$ be a generating set of Hamiltonians. For each boundary $s \in \{-,+\}$ such that $\partial_s \in \bR \cap E$ and each direction $d \in \{-,+\}$  either the weak or the strong condition holds: 
	\begin{description}
		\item[Strong condition] There is a closed interval of positive length $K \subseteq E$ and $\partial_s \in K$ such that
		\begin{equation*}
		\lim_{p \rightarrow d \infty} \inf_{x \in K}  \cH_\emptyset(x,p) = \infty.
		\end{equation*}
		\item[Weak condition] 
		We have either (1) or (2):
		\begin{enumerate}[(1)]
			\item There is a closed interval of positive length $K \subseteq E$ and $\partial_s \in K$ and a continuous function $h_{s}^{d}$ on $K$ such that
			\begin{equation*}
			\lim_{p \rightarrow d \infty} \sup_{x \in K} \left|\cH_\emptyset(x,p) - h_{s}^{d}(x) \right| = 0.
			\end{equation*}
			\item For $x_\alpha,y_\alpha \rightarrow \partial_s$ and $p_\alpha = \alpha(x_\alpha-y_\alpha) \rightarrow d \infty$, we have
			\begin{equation*}
			\liminf_{\alpha \rightarrow \infty} \cH_\emptyset(x_\alpha,p_\alpha) - \cH_\emptyset(y_\alpha,p_\alpha) \leq 0.
			\end{equation*}
		\end{enumerate}
	\end{description}
\end{condition}

\begin{remark}
	Note that weak condition (1) implies weak condition (2), but is often easier to check. For the Hamiltonian $\cH(x,p) = x\left[e^{-p} - 1\right]$ on the space $[0,\infty)$ and $s = -, d = -$, condition (2) holds as
	\begin{equation*}
	\cH_\emptyset(x_\alpha,\alpha(x_\alpha - y_\alpha)) - \cH_\emptyset(y_\alpha,\alpha(x_\alpha - y_\alpha)) = \left[x_\alpha - y_\alpha \right]\left[e^{-\alpha(x_\alpha - y_\alpha)} - 1\right] \leq 0
	\end{equation*}
	due to the anti-symmetry of the occurrence of the difference $x_\alpha - y_\alpha$. The strong condition, nor the weak condition (1) holds.
\end{remark}

\begin{remark}
	Consider the Hamiltonian corresponding to the Yule process: $\cH(x,p) = x\left[e^{p} - 1\right]$ on the space $[0,\infty)$. It follows that Condition \ref{condition:boundary} fails for $d = +$, corresponding with our discussion in Section \ref{subsection:failure_of_LDP}
\end{remark}

\begin{theorem} \label{theorem:comparison_principle_1d}
	Let $\bH := \{\cH_\emptyset\} \cup \{\cH_z\}_{z \in \{-,+\}}$ be a generating set of Hamiltonians. Suppose that $\bH$ has a good containment function $\Upsilon$ and satisfies Conditions \ref{condition:basic} and \ref{condition:boundary}. Let $h \in C_b(E)$ and $\lambda > 0$. 
	
	Then the comparison principle holds for subsolutions to $f - \lambda H_\dagger f = h$ and supersolutions to $f - \lambda H_\ddagger f = h$.
\end{theorem}

\begin{proof}[Proof of Theorem \ref{theorem:comparison_principle_1d}]
	We verify the conditions for Proposition \ref{proposition:comparison_conditions_on_H} using the penalization function $\Psi(x,y) = \frac{1}{2} (x-y)^2$.
	
	Thus, for $\alpha,\varepsilon >0$ let $x_{\alpha,\varepsilon},y_{\alpha,\varepsilon} \in E$ be such that
	\begin{multline*} 
	\frac{u(x_{\alpha,\varepsilon})}{1-\varepsilon} - \frac{v(y_{\alpha,\varepsilon})}{1+\varepsilon} -  \Psi_\alpha(x_{\alpha,\varepsilon},y_{\alpha,\varepsilon}) - \frac{\varepsilon}{1-\varepsilon}\Upsilon(x_{\alpha,\varepsilon}) -\frac{\varepsilon}{1+\varepsilon}\Upsilon(y_{\alpha,\varepsilon}) \\
	= \sup_{x,y \in E} \left\{\frac{u(x)}{1-\varepsilon} - \frac{v(y)}{1+\varepsilon} - \Psi_\alpha(x,y)  - \frac{\varepsilon}{1-\varepsilon}\Upsilon(x) - \frac{\varepsilon}{1+\varepsilon}\Upsilon(y)\right\}.
	\end{multline*}
	We fix $\varepsilon > 0$ and prove that
	\begin{equation}\label{condH:negative:liminf_explicit}
	\liminf_{\alpha \rightarrow \infty} \cH_\dagger\left(x_{\alpha,\varepsilon},\alpha\nabla \Psi(\cdot,y_{\alpha,\varepsilon})(x_{\alpha,\varepsilon})\right) - \cH_\ddagger\left(y_{\alpha,\varepsilon},\alpha\nabla \Psi(\cdot,y_{\alpha,\varepsilon})(x_{\alpha,\varepsilon})\right) \leq 0.
	\end{equation}
	This implies that \eqref{condH:negative:liminf} is satisfied and the comparison principle holds.
	
	\smallskip
	
	From this point onward, we drop $\varepsilon > 0$ from our notation. In addition, we write $p_\alpha := \alpha\nabla \Psi(\cdot,y_{\alpha})(x_{\alpha}) = \alpha(x_\alpha - y_\alpha)$ to shorten formula's below. By Lemma \ref{lemma:doubling_lemma}, the sequences $\{x_\alpha,y_\alpha\}_{\alpha >0}$ are contained in some compact set and all limit points are of the form $(z,z)$. Thus, without loss of generality, we can assume that $(x_\alpha,y_\alpha) \rightarrow (z,z)$. The proof proceeds depending on whether $z$ is in the interior or is a boundary point. 
	
	\smallskip
	
	\textit{Step 1:} Suppose that $z \in E^\circ$. Then there is some compact set $K \subseteq E^\circ$ with $z \in K$.

	Without loss of generality, we can assume that $(x_\alpha,y_\alpha)$ contains a subsequence such that $p_\alpha \geq 0$. We will denote the subsequence also by $\alpha$. We argue using Condition \ref{condition:basic} holds for $d = +$ (In the case we have a subsequence with $p_\alpha \leq 0$, we would argue with $d = -$). 
	
	Suppose that Condition \ref{condition:basic} (1) holds. Then it follows by Lemma \ref{lemma:control_on_H} that $\sup_{\alpha} |p_\alpha| < \infty$. In this case, we can extract a converging subsequence with index $\hat{\alpha}$, which implies, using the continuity of $\cH_\emptyset$ that 
	\begin{equation*}
	\lim_{\hat{\alpha} \rightarrow \infty} \cH_\dagger\left(x_{\hat{\alpha}},p_{\hat{\alpha}}\right) -  \cH_\dagger\left(y_{\hat{\alpha}},p_{\hat{\alpha}}\right) = \lim_{\hat{\alpha} \rightarrow \infty} \cH_\emptyset\left(x_{\hat{\alpha}},p_{\hat{\alpha}}\right) -  \cH_\emptyset\left(y_{\hat{\alpha}},p_{\hat{\alpha}}\right)  = 0
	\end{equation*}
	implying \eqref{condH:negative:liminf_explicit}.
	
	Suppose that Condition \ref{condition:basic} (2) holds. If as above, the sequence $p_\alpha$ contains a bounded subsequence, then \eqref{condH:negative:liminf_explicit} holds. If, however, $p_\alpha \rightarrow \infty$, then
	\begin{equation*}
	\lim_{\hat{\alpha} \rightarrow \infty} \cH_\dagger\left(x_{\hat{\alpha}},p_{\hat{\alpha}}\right) -  \cH_\dagger\left(y_{\hat{\alpha}},p_{\hat{\alpha}}\right) = \lim_{\hat{\alpha} \rightarrow \infty} \cH_\emptyset\left(x_{\hat{\alpha}},p_{\hat{\alpha}}\right) -  \cH_\emptyset\left(y_{\hat{\alpha}},p_{\hat{\alpha}}\right)  = h^+(z) - h^+(z) = 0
	\end{equation*}
	implying \eqref{condH:negative:liminf_explicit}.
	
	\textit{Step 2:} Suppose that $z = \partial_-$. The case of $z = \partial_+$ is similar. We show that we can argue using the operator $\cH_\emptyset$ instead of $\cH_\dagger$ and $\cH_\ddagger$. By Condition \ref{condition:basic}, we find that $p_\alpha \geq 0$ implies $\cH_{-}(\partial_-,p_\alpha) \geq \cH_\emptyset(\partial_-,p_\alpha)$, so that $\cH_\ddagger(y_\alpha,p_\alpha) = \cH_\emptyset(y_\alpha,p_\alpha)$. In addition, if $p_\alpha = 0$, then $\cH_\dagger(x_\alpha,p_\alpha) = \cH_\emptyset(x_\alpha,p_\alpha)$, whereas if $p_\alpha > 0$, then $x_\alpha > y_\alpha \geq \partial_-$, so that also $\cH_\dagger(x_\alpha,p_\alpha) = \cH_\emptyset(x_\alpha,p_\alpha)$. Similar arguments treat the case $p_\alpha \leq 0$. We conclude that for all $\alpha > 0$ we have 
	\begin{equation*}
	\cH_\dagger(x_\alpha,p_\alpha) = H_\emptyset(x_\alpha,p_\alpha), \qquad \cH_\ddagger(y_\alpha,p_\alpha) = \cH_\emptyset(y_\alpha,p_\alpha).
	\end{equation*}
	This implies we can argue by using $\cH_\emptyset$ instead of $\cH_\dagger,\cH_\ddagger$, which is analogous as to the proof in step 1.
\end{proof}

To conclude this section on the comparison principle for one-dimensional systems, we show that satisfying Conditions \ref{condition:basic} and \ref{condition:boundary} is stable under the formation of finite sums. This makes checking the conditions easy: if it holds for easily identifiable building blocks of a Hamiltonian, then it holds for the sum of these Hamiltonians.

\begin{definition} \label{definition:addition_of_generating_sets_of_Hamiltonians}
	Let $\bH_i := \{\cH_{J,i}\}_{J \in \cJ(E)}$, $i \in \{1,\dots,k\}$ be generating sets of Hamiltonians. We define $\bH := \sum_{i=1}^k \bH_i$ to be the set of Hamiltonians $\{\cH_J\}_{J \in \cJ(E)}$, where $\cH_J = \sum_{i=1}^k \cH_{J,k}$. Note that $\bH$ is also a generating set of Hamiltonians.
\end{definition}

\begin{lemma} \label{lemma:additive_structure_conditions_one_d_hamiltonians}
	Let $\bH_i := \{\cH_{\emptyset,i}\} \cup \{\cH_{z,i}\}_{z \in \{-,+\}}$, $i \in \{1,\dots,k\}$ be a generating sets of Hamiltonians that satisfy Condition \ref{condition:basic}. Suppose that for each $s \in \{-,+\}$ with $\partial_s \in E \cap \partial E$ we have (a) or (b):
	\begin{enumerate}[(a)]
		\item $\cH_{\emptyset,1},\dots,\cH_{\emptyset,k}$ satisfy Condition \ref{condition:boundary}.
		\item There is $i \in \{1,\dots,k\}$ such that $\cH_{\emptyset,i}$ satisfies strong Condition \ref{condition:boundary}.
	\end{enumerate}
	Then $\bH := \bH_1 + \dots + \bH_k$ satisfies Conditions \ref{condition:basic} and \ref{condition:boundary}.
\end{lemma}

\begin{proof}
	All claims are immediate.
\end{proof}

\subsection{The comparison principle for a class of multi-dimensional systems} \label{section:comparison_multi_d}

We proceed with giving easily verifiable conditions for a class of Hamiltonians on $E := [0,\infty)^d$ that corresponds to the large deviations of $d$ interacting species. Regarding the boundary structure that we introduced in Section \ref{section:boundary_structure}, we write $E_i=\left\{x\in E, x_i=0\right\}$ for all $i$. For $J \subseteq \{1,\dots,d\}$, we write $E_J := \bigcap_{i \in J} E_i$.

\smallskip

The dynamics that we will consider consists of two parts. The first part arises from interaction between individuals. This leads to a base Hamiltonian $\cH_0$ that vanishes on the boundary. We will consider a second part that consists of immigration of individuals, and a second part that consists of harvesting. This final part introduces a discontinuity of the dynamics at the boundary.

\begin{condition} \label{condition:boundary_multi_d}
	We have a continuous map $\cH_0(x,p)$ such that $p \mapsto \cH_0(x,p)$ is convex for each $x \in E$. In addition, we have for each $i \in \{1,\dots,d\}$ positive continuous functions $a_i,b_i : E \rightarrow \bR$.
	
	$\cH_0$ and $a_i,b_i$ satisfy either (a) or (b):
	\begin{enumerate}[(a)]
		\item \begin{enumerate}[(i)]
			\item For each compact $K \subseteq E$ there is a $M_K \geq 0$ such that
			\begin{equation*}
			\cH_0(x,p) \geq -|p| M_K.
			\end{equation*}
			\item For each $i \in \{1,\dots,d\}$ the functions $a_i,b_i$ are strictly positive continuous. 
		\end{enumerate}
		\item \begin{enumerate}[(i)]
			\item For each compact set $K \subseteq E^\circ$ there is some $M_K \geq 0$ such that 
			\begin{equation*}
			\lim_{|p|\rightarrow \infty} \inf_{x \in K} \frac{\cH(x,p)}{|p|} \geq M_K.
			\end{equation*}
			\item The functions $a_i,b_i$ are strictly positive on a neighbourhood of $\partial E$.
		\end{enumerate}
	\end{enumerate}
\end{condition}

The difference in conditions (a) and (b) is as follows: (b) requires at least linear growth of $\cH_0$ but puts no conditions on $a_i,b_i$ on the interior of $E$, whereas (a) relaxes the growth of $\cH_0$ and puts conditions on $a_i,b_i$ on the interior instead.

\begin{remark}
	Note that a map $\cH_0$ that is continuously differentiable immediately satisfies the lower bound in (a).(i) as $\cH_0(x,p) \geq |p| |\partial_p \cH_0(x,0)|$.
\end{remark}

We then consider the generating set of Hamiltonians $\bH$ defined as
\begin{align*}
\cH_\emptyset(x,p) & = \cH_0(x,p)+\sum_{k=1}^d a_k(x)(e^{p_k}-1)+b_k(x)(e^{-p_k}-1)\\
\intertext{and for $J\in\cJ(E)$ and $x\in E_J$}
\cH_J(x,p) & =  \cH_0(x,p)+\sum_{k\in J} a_k(x)(e^{p_k}-1) + \sum_{l\notin J}a_l(x)(e^{p_l}-1)+b_l(x)(e^{-p_l}-1).
\end{align*}

\begin{theorem}\label{theorem:comparison_principle_multid}
	Let Condition \ref{condition:boundary_multi_d}  be satisfied for $\bH$ and suppose there is a good containment function $\Upsilon$ for $\bH$. Then the comparison principle holds for subsolutions to $f - \lambda H_\dagger f = h$ and supersolutions to $f - \lambda H_\ddagger f = h$.
\end{theorem}

\begin{proof}
	Before starting, note that Condition \ref{condition:boundary_multi_d} implies that for any compact set $K\subseteq E$, $\forall\  1\leq j \leq d$,
	\begin{equation} \label{equation:strong_condition_ddimension}
	\lim_{|p|\rightarrow \infty}\inf_{x\in K} \cH_\emptyset(x,p)=\infty.
	\end{equation}
	
	To establish the comparison principle, we use Proposition \ref{proposition:comparison_conditions_on_H} with penalization function $\Upsilon(x,y) = \frac{1}{2} |x-y|^2$. Thus, for $\alpha,\varepsilon >0$ let $x_{\alpha,\varepsilon},y_{\alpha,\varepsilon} \in E$ be such that
	\begin{multline*} 
	\frac{u(x_{\alpha,\varepsilon})}{1-\varepsilon} - \frac{v(y_{\alpha,\varepsilon})}{1+\varepsilon} -  \Psi_\alpha(x_{\alpha,\varepsilon},y_{\alpha,\varepsilon}) - \frac{\varepsilon}{1-\varepsilon}\Upsilon(x_{\alpha,\varepsilon}) -\frac{\varepsilon}{1+\varepsilon}\Upsilon(y_{\alpha,\varepsilon}) \\
	= \sup_{x,y \in E} \left\{\frac{u(x)}{1-\varepsilon} - \frac{v(y)}{1+\varepsilon} - \Psi_\alpha(x,y)  - \frac{\varepsilon}{1-\varepsilon}\Upsilon(x) - \frac{\varepsilon}{1+\varepsilon}\Upsilon(y)\right\}.
	\end{multline*}
	We fix $\varepsilon > 0$ and prove that
	\begin{equation}\label{condH:negative:liminf_explicit_multid}
	\liminf_{\alpha \rightarrow \infty} \cH_\dagger\left(x_{\alpha,\varepsilon},\alpha\nabla \Psi(\cdot,y_{\alpha,\varepsilon})(x_{\alpha,\varepsilon})\right) - \cH_\ddagger\left(y_{\alpha,\varepsilon},\alpha\nabla \Psi(\cdot,y_{\alpha,\varepsilon})(x_{\alpha,\varepsilon})\right) \leq 0.
	\end{equation}
	Henceforth, we drop the $\varepsilon$ from the notation. By Lemma \ref{lemma:doubling_lemma} the sequences $(x_\alpha,y_\alpha)$ are contained in some compact set with limit points of the form $(z,z)$. Up to a subsequence we can suppose $\{x_\alpha,y_\alpha\}\rightarrow (z,z)$.  Write 
	\begin{equation*}
	p_\alpha = (p_{\alpha,1},p_{\alpha,2}....,p_{\alpha,d}) := \alpha\nabla \Psi(\cdot,y_{\alpha})(x_{\alpha}). 
	\end{equation*}
	The rest of the proof depends on whether  $z$ belongs to the boundary or not.

	\textit{Step 1}: Suppose that $z\in E^\circ$. Then let $K\subseteq E^\circ$ a compact set with $z\in K$. Lemma \ref{lemma:control_on_H} combined with \eqref{equation:strong_condition_ddimension}  ensures that $\forall \ 1\leq j\leq d$, $|p_{\alpha,j}|<\infty$. Thus going to a subsequence such that $p_\alpha$ converges, continuity of $\cH_\emptyset$ implies that \eqref{condH:negative:liminf_explicit_multid} holds.
	
	\textit{Step 2}: Suppose that $z\in \partial E$. Note that as $p_{\alpha,j} =  \alpha(x_{\alpha,j}-y_{\alpha,j})$ for each $j$:
	\begin{equation*}
	\left\{1\leq j \leq d \, \middle| \, x_{\alpha,j} = 0\right\} \subseteq \left\{1\leq j \leq d \, \middle| \, p_{\alpha,j} \leq 0\right\}.
	\end{equation*}	
	Next, let $J \in \cJ(E)$ such that $x_\alpha \in E_J$. We find that $p_{\alpha,j}\leq 0$ and, hence,
	\begin{equation*}
	\cH_\emptyset(x_\alpha,p_\alpha)-\cH_J(x_\alpha,p_\alpha)=\sum_{j\in J}b_j(x_\alpha)\left(e^{-p_{\alpha,j}}-1\right)\geq 0.
	\end{equation*}
	As this holds for all $J$ with $x_\alpha \in E_J$, we find $\cH_\dagger(x_\alpha,p_\alpha) = \max_{J\in\cJ(E):x_\alpha \in E_J} \cH_J(x_\alpha,p_\alpha) = \cH_\emptyset(x_\alpha,p_\alpha)$. Similarly, one establishes $\cH_\ddagger(y_\alpha,p_\alpha) = \cH_\emptyset(y_\alpha,p_\alpha)$. 
	
	Thus \eqref{condH:negative:liminf_explicit_multid} follows as in step 1. 
\end{proof}

\section{Application to large deviations} \label{section:abstract_LDP}

In the extensive monograph, \cite{FK06}, Feng and Kurtz introduced a general method to prove large deviations with speed $r_n$ for a sequence of Markov processes $\{X_n\}_{n \geq 1}$, having generators $\{ A_n \}_{n \geq 1}$, via the well-posedness of associated Hamilton-Jacobi equations. See \cite{Kr19,Kr19c} for a recent new proof of this result. The Feng-Kurtz strategy is based on three observations.
\begin{enumerate}[(1)]
	\item By an extended variant of the projective limit theorem, the large deviation principle for processes follows from exponential tightness combined with large deviation principles for the finite dimensional distributions.
	\item The large deviation principle for finite dimensional distributions is established by a variant of Bryc's theorem, using the Markov property of the processes. It suffices to assume large deviations for time $0$ and that the semigroup of conditional log-Laplace transforms 
	\begin{equation*}
	V_n(t)f(x) = \frac{1}{r_n} \log \bE\left[e^{r_nf(X_n(t))} \, \middle| X_n(0) = x \right]
	\end{equation*}
	converge to a limiting semigroup $V(t)$.
	\item Using control theory, the limiting semigroup can be rewritten as a variational semigroup, which allows to rewrite the rate-function in Lagrangian form.
\end{enumerate}

The convergence of semigroups for (2) uses non-linear semigroup theory. We discuss the connection of the convergence of semigroups to the comparison principles that we studied in Section \ref{section:framework_HJ_boundary}.

For each $n$, we distinguish three objects of importance:
\begin{itemize}
	\item The semigroup $\{V_n(t)\}_{t \geq 0}$;
	\item The formal generator $H_n = \frac{\dd}{\dd t} |_{t = 0} V_n(t) f$ of the semigroup $\{V_n(t)\}_{t \geq 0}$ is given by $H_nf = r_n^{-1} e^{-nf} A_n e^{nf}$;
	\item The resolvents $\{R(\lambda)\}_{\lambda > 0}$, formally given by $R(\lambda) = (\bONE- \lambda H_n)^{-1}$.
\end{itemize}
The main step in the Trotter-Kato-Kurtz proof of the convergence of semigroups is that the convergence of resolvents implies the convergence of semigroups. This result is usually extended with the fact that the convergence of generators, and a limiting generator that is sufficiently big, implies the convergence of resolvents.

\smallskip

In our setting, it is not clear how to directly construct a limiting generator as the coefficients of the generator can be discontinuous at the boundary. As a consequence, it is also not clear how to establish the convergence of resolvents directly.

The theory of viscosity solutions provides us with an alternative. \cite{FK06} introduce two operators $H_\dagger \subseteq ex-\subLIM_n H_n$, $H_\ddagger \subseteq ex-\superLIM_n H_n$ that serve as a limiting `upper' and `lower' bound for the operators $H_n$, see Definition \ref{definition:subLIM_superLIM} below. These upper and lower bounds can be used to obtain upper and lower bounds for the limiting resolvent: let $h_n$ converge boundedly and uniformly on compacts to a function $h$, then one can show that the functions $\overline{f},\underline{f}$ given by
\begin{equation} \label{eqn:semi_relaxed_limits}
\begin{aligned}
\overline{f}(x) & := \sup\left\{\limsup_{n \rightarrow \infty} R_n(\lambda)h_n(x_n) \, \middle| \, x_n \rightarrow x\right\}, \\
\underline{f}(x) & := \inf\left\{\liminf_{n \rightarrow \infty} R_n(\lambda)h_n(x_n) \, \middle| \, x_n \rightarrow x\right\},
\end{aligned}
\end{equation}
are a viscosity subsolution to $f - \lambda H_\dagger f = h$ and a viscosity supersolution to $f - \lambda H_\ddagger f = h$ respectively. Thus, if the comparison principle is satisfied, there is a unique limiting function $\overline{f} = \underline{f}$, which we will denote by $R(\lambda)h$. By the Crandall-Liggett theorem, \cite{CL71}, these resolvents generate a semigroup
\begin{equation*}
V(t)f = \lim_{n \rightarrow \infty} R\left(\frac{t}{n},H\right)^n f,
\end{equation*}
and by the Trotter-Kurtz approximation theorem we have $V_n(t) \rightarrow V(t)$. 

To conclude this discussion: the comparison principle is sufficient to establish (2), which, combined with a standard verification of exponential tightness (1), is sufficient to establish path-space large deviations of a sequence of Markov processes. 

\smallskip

Before discussing this results in Section \ref{section:synthesis_LDP}, we introduce control theory in Section \ref{section:control_theory}, so that (3) leads us to a Lagrangian form of the rate function. Using control theory, we follow \cite{FK06} in introducing a variational semigroup $\{\bfV(t)\}_{t \geq 0}$, as well as a variational resolvent $\{\bfR(\lambda)\}_{\lambda >0}$. It can be shown, cf. \cite[Theorem 8.27]{FK06} that the variational resolvent also yields subsolutions to $f - \lambda H_\dagger f = h$ and supersolutions to $f - \lambda H_\ddagger f = h$. Thus, the comparison principle establishes that the variational resolvent must equal the resolvent $R(\lambda)$. This in turn establishes the equality of semigroups, which leads to a Lagrangian form of the rate function.

\subsection{Control theory} \label{section:control_theory}

Let $E$ be a $d$ dimensional convex Polyhedron, let $\bH = \{\cH_J\}_{J \in \cJ(E)}$ be a generating set of Hamiltonians and let $\cL_J : E_J \times \bR^d \rightarrow [0,\infty]$ and $\cL : E \times \bR^d \rightarrow [0,\infty]$ the corresponding Lagrangians as defined in Definition \ref{definition:construction_of_L}. Using $\cL$ we introduce a variational semigroup and resolvent:
\begin{align*}
\bfV(t)f(x) & := \sup_{\substack{\gamma \in \cA\cC, \\ \gamma(0) = x}} \left\{f(\gamma(t)) - \int_0^t \cL(\gamma(s),\dot{\gamma}(s)) \dd  \right\}, \\
\bfR(\lambda) f(x) & := \limsup_{t \rightarrow \infty} \sup_{\substack{\gamma \in \cA\cC \\ \gamma(0) = x}} \left\{ \int_0^t \lambda^{-1} e^{-\lambda^{-1}t} \left(f(\gamma(t)) - \int_0^s \cL(\gamma(r),\dot{\gamma}(r)) \dd r \right) \dd s \right\}.
\end{align*}

In the next two propositions, we establish the conditions that are needed for the application of the control theory component of \cite[Theorem 8.27]{FK06}. The first result can be used to establish the path-space compactness of the set of trajectories that start in a compact set and have uniformly bounded Lagrangian cost. The compactness of this set can be used to establish various properties of $\bfV$ and $\bfR$. The second result is crucial in establishing that the lower semi-continuous regularization of $\bfR(\lambda)h$ is a viscosity supersolution to the Hamilton-Jacobi equation $f - \lambda H_\ddagger f = h$. 

These properties are proven in \cite{FK06}, where the results of the propositions below are taken as an input for the theory. The outcomes of these are used to establish the variational expression of the rate-function in Theorem \ref{theorem:abstract_LDP}.

\begin{proposition} \label{proposition:FK8.9}
	Let $\{\cH_J\}_{J \in \cJ(E)}$ be a generating set of Hamiltonians. Suppose there is a good containment function $\Upsilon$ for $\{\cH_J\}_{J \in \cJ(E)}$. Then we have
	\begin{enumerate}[(a)]
		\item $\cL : E \times \bR^d \rightarrow [0,\infty]$ is lower semi-continuous and for each compact set $K \subseteq E$ and $c \in \bR$ the set
		\begin{equation*}
		\left\{(x,v) \in K \times \bR^d \, \middle| \,  \cL(x,v) \leq c \right\}
		\end{equation*}
		is compact in $E \times \bR^d$
		\item For each compact $K \subseteq E$, $T > 0$ and $0 \leq M < \infty$, there exists a compact set $K' = K'(K,T,M) \subseteq E$ such that $\gamma \in \cA\cC$ and $\gamma(0) \in K$ and
		\begin{equation*}
		\int_0^T \cL(\gamma(s),\dot{\gamma}(s)) \, \dd s \leq M
		\end{equation*}
		implies $\gamma(t) \in K'$ for all $0 \leq t \leq T$.
		\item For each $f \in C_b^2(E)$ and compact $K \subseteq E$, there exists a right-continuous non-decreasing function $\psi_{f,K} : \bR^+ \rightarrow \bR^+$ such that $\lim_{r \rightarrow \infty} r^{-1} \psi_{f,K}(r)  = 0$ and
		\begin{equation*}
		|\ip{\nabla f(x)}{v}| \leq \psi_{f,K}(\cL(x,v)), \qquad \forall (x,v) \in E \times \bR^d, x \in K. 
		\end{equation*}
	\end{enumerate}
\end{proposition}

\begin{proof}
	(a) $\cL : E \times \bR^d \rightarrow [0,\infty]$ is lower semi-continuous as it is the Legendre transform of $H_\dagger$. The compactness of the level sets follows as in the proof of Lemma 2 in \cite{Kr16b}.
	
	(b) The proof is a standard Lyapunov function proof. We follow \cite{CoKr17}. Let $\gamma$ be absolutely continuous with Lagrangian cost bounded by $M$ and $\gamma(0) \in K$. Let $t \leq T$, then
	\begin{align*}
	\Upsilon(\gamma(t)) & = \Upsilon(\gamma(0)) + \int_0^t \ip{\nabla\Upsilon(\gamma(s))}{\dot{\gamma}(s)} \dd s \\
	& \leq \Upsilon(\gamma(0)) + \int_0^t \cL(\gamma(s),\dot{\gamma}(s)) + \cH(\gamma(s),\nabla \Upsilon(\gamma(s))) \dd s \\
	& \leq \sup_{y \in K} \Upsilon(y) + M + \int_0^T \sup_z \sup_{J: z \in E_J}  \cH_J(z,\nabla \Upsilon(z)) \dd s \\
	& =: C < \infty.
	\end{align*}
	Thus, we can take $K'= \{z \in E \, | \, \Upsilon(z) \leq C\}$.
	
	(c) follows as in the proof of Lemma 10.21 in \cite{FK06}. Note that as the $H_J's$ are continuous, we have
	\begin{equation*}
	\overline{\cH}_K(c) = \sup_{|p|\leq c} \sup_{x \in K} \cH_\dagger(x,p) \leq  \sup_{|p|\leq c} \sup_{x \in K} \max_{J: \, x \in E_J} \cH_J(x,p) < \infty,
	\end{equation*}
	which is an essential ingredient for the proof in \cite{FK06}.
\end{proof}

\begin{proposition}\label{proposition:FK8.11}
	Let $\{\cH_J\}_{J \in \cJ(E)}$ be a generating set of Hamiltonians. Suppose there is a good containment function $\Upsilon$ for $\{\cH_J\}_{J \in \cJ(E)}$. 
	
	Then there exists for each $x \in E$ and $f \in C_c^2(E)$ a $\gamma \in \cA\cC$ with $\gamma(0) = x$ such that for all $t_1 < t_2$:
	\begin{equation*}
	\int_{t_1}^{t_2} \cH_\ddagger f(\gamma(s)) \dd s \leq \int_{t_1}^{t_2} \ip{\dd f(\gamma(s))}{\dot{\gamma}(s)} - \cL(\gamma(s),\dot{\gamma}(s)) \, \dd s.
	\end{equation*}
\end{proposition}

The proof of the proposition uses the theory of differential inclusions. In Appendix \ref{appendix:differential_inclusions}, we give the relevant basic definitions and a result that establishes the existence of a solution to `well behaved' differential inclusions.

\begin{proof}
	Pick $f \in C_c^2(E)$ and $x \in E$.
	
	\textit{Step 1:} Suppose we have a solution $\gamma_f$ to the differential inclusion
	\begin{equation} \label{eqn:proof_8.11_diff_inclusion}
	\dot{\gamma}_f(s) \in F_f(x) := \text{ch} \bigcup_{J : x \in E_J} \partial_p \cH_J(\gamma_f(s),\nabla f(\gamma_f(s))), \qquad \gamma_f(0) = x.
	\end{equation}
	We show that 
	\begin{equation} \label{eqn:proof_8.11_bound}
	\int_{t_1}^{t_2} \cH_\ddagger f(\gamma_f(s)) \dd s \leq \int_{t_1}^{t_2} \ip{\nabla f(\gamma_f(s))}{\dot{\gamma}_f(s)}  - \cL(\gamma_f(s),\dot{\gamma}_f(s)) \,\dd s.
	\end{equation}
	It suffices to show that for almost every $s$, we have ${\nabla f(\gamma_f(s))}{\dot{\gamma}_f(s)}  - \cL(\gamma_f(s),\dot{\gamma}_f(s)) \geq \cH_\ddagger f(\gamma_f(s))$. In particular, it suffices to do this for all times $s$ at which $\dot{\gamma_f(s)} \in F_f(\gamma_f(s))$.
	
	Thus, suppose $v \in F_f(y)$. Then there are $\lambda_J$ and $v_J = \partial_p \cH_J(y,\nabla f(y))$ such that $v = \sum \lambda_J v_J$. As $\cL(y,v) \leq \widehat{\cL}(y,v) \leq \sum \lambda_J \cL_J(x,v_J)$:
	\begin{align*}
	\ip{\nabla f(y)}{v} - \cL(y,v) & \geq \sum \lambda_J \left( \ip{\nabla f(y)}{v_J} - \cL_J(y,v_J) \right) \\
	& \geq \sum \lambda_J \cH_J(y,\nabla f(y)) \\
	& \geq \cH_\ddagger(y,\nabla f(y)),
	\end{align*} 	
	establishing \eqref{eqn:proof_8.11_bound}.

	\textit{Step 2:} We construct a solution to the differential inclusion \eqref{eqn:proof_8.11_diff_inclusion}, for which we use Lemma \ref{lemma:solve_differential_inclusion}. Note that the Lemma assumes growth bounds on the size of $F_f$, that are not necessarily satisfied.
	
	\textit{Step 2a:} Therefore, we start with \textit{a priori} control on the range of a solution. Let, as in step 1, $\gamma_f$ be a solution to \eqref{eqn:proof_8.11_diff_inclusion}. \eqref{eqn:proof_8.11_bound} implies
	\begin{equation*}
	\int_0^T \cL(\gamma_f(s),\dot{\gamma}_f(s)) \dd s \leq 	\int_{0}^{T} \ip{\nabla f(\gamma_f(s))}{\dot{\gamma}_f(s)}  - \cH_\ddagger f(\gamma_f(s)) \,\dd s.
	\end{equation*}
	Note that as $f$ has compact support in $E$, the map $x \mapsto \cH_\ddagger(x,\nabla f(x))$ is bounded from below as all $\cH_J$ are continuous, and $F_f$ is bounded on compact sets, there is some $M \geq 0$ such that
	\begin{equation*}
	\int_0^T \cL(\gamma_f(s),\dot{\gamma}_f(s)) \dd s \leq 	M.
	\end{equation*}
	We conclude by Proposition \ref{proposition:FK8.9} (b) that the trajectory $\gamma_f$ remains in some compact set $K' \subseteq E$. 
	
	\textit{Step 2b:} By step 2a, it suffices to construct a solution to $\dot{\gamma} \in \hat{F}_f(\gamma)$, where $\hat{F}_f$ equals $F_f$ on $K'$, and is equal to $0$ outside a neighbourhood of $K'$. For example, we can smoothly multiply $F_f$ by a cut-off function.
	
	Thus, we construct a solution using Lemma \ref{lemma:solve_differential_inclusion}. We will verify the conditions for $F_f$. Note that the modification above can be made such to preserve the conditions for this lemma.
	
	We check the conditions (a)-(d). (a) is clear by definition. For (b), fix $x \in E$ and a neighbourhood $\cU$ of $F(x)$. We write $J^* := J^*(x)$. As for each $J$ the set $E_J$ is closed in $E$ by Definition \ref{definition:tangent_cones}, we can choose a neighbourhood $\cV$ of $x$ that is contained in the $E$ interior of $E_{J^*}$. This implies that for $y \in \cV_0$ the set $F_f(y)$ is the convex hull of $\partial_p \cH_J(y,\nabla f(y))$ with $J \subseteq J^*$. Because $\partial_p \cH_J$ is continuous for each such $J$ by Definition \ref{definition:generating_Hamiltonians} (b), we can find a smaller neighbourhood $\cV \subseteq \cV_0$ with $x \in \cV$ so that $F_f(y) \subseteq \cU$ for all $y \in \cV$, establishing that $F_f$ is upper semi-continuous.
	
	For (c), first note that $T_E(x) = T_x E$. Let $J^* := J^*(x)$. By Definition \ref{definition:generating_Hamiltonians} (c), we have $\partial_p \cH_{J^*}(x,\nabla f(x)) \subseteq T_x E$ implying that $F_f(x) \cap T_x E \neq \emptyset$.

	Finally, (d) follows from Definition \ref{definition:generating_Hamiltonians} (d).
	
	Thus, by Lemma \ref{lemma:solve_differential_inclusion} there is a solution $\gamma_f$ to the differential inclusion \eqref{eqn:proof_8.11_diff_inclusion}.
\end{proof}

\subsection{Synthesis: a general large deviation principle} \label{section:synthesis_LDP}

To connect the Hamilton-Jacobi equation to the large deviation principle, we introduce some additional concepts. For $E \subseteq \bR^d$ and $E_n \subseteq E$ for all $n$, we write $E = \lim_{n \to \infty} E_n$ if for every $x \in E$ there exist $x_n \in E_n$ such that $x_n \rightarrow x$. 

For $f_n \in C_b(E_n)$ and $f \in C_b(E)$, we write $\LIM f_n = f$ if $\sup_n \vn{f_n} < \infty$ and if for each compact $K \subseteq E$, we have
\begin{equation*}
\lim_{n \rightarrow \infty} \sup_{x \in K \cap E_n} \left|f_n(x) - f(x) \right| = 0.
\end{equation*}

\begin{definition}[Condition 7.11 in \cite{FK06}] \label{definition:subLIM_superLIM}
	Suppose that for each $n$ we have an operator $H_{n} \subseteq C_b(E_n) \times C_b(E_n)$. Let $(v_n)_{n \geq 0}$ be a sequence of real numbers such that $v_n \uparrow \infty$.
	
	The \textit{extended sub-limit}, denoted by $ex-\subLIM_n H_{n}$, is defined by the collection $(f,g) \in C_l(E)\times \cM_u(E,\overline{\bR})$ for which there exist $(f_n,g_n) \in H_{n}$ such that
	\begin{gather} 
	\LIM f_n \wedge c = f \wedge c, \qquad \forall \, c \in \bR, \label{eqn:convergence_condition_sublim_constants} \\
	\sup_n \frac{1}{v_n} \log \vn{g_n} < \infty, \qquad \sup_{n} \sup_{x \in E_n} g_n(x) < \infty, \label{eqn:convergence_condition_sublim_uniform_gn}
	\end{gather}
	and that, for every compact set $K \subseteq E$ and every sequence $z_n \in K \cap E_n$ satisfying $\lim_n z_n = z$ and $\lim_n f_n(z_n) = f(z) < \infty$, we have
	\begin{equation} \label{eqn:sublim_generators_upperbound}
	\limsup_{n \uparrow \infty}g_{n}(z_n) \leq g^*(z).
	\end{equation}
	The \textit{extended super-limit}, denoted by $ex-\superLIM_n H_{n}$, is defined by the collection $(f,g) \in C_u(E)\times \cM_l(E,\overline{\bR})$ for which there exist $(f_n,g_n) \in H_{n}$ such that
	\begin{gather} 
	\LIM f_n \vee c = f \vee c, \qquad \forall \, c \in \bR, \label{eqn:convergence_condition_superlim_constants} \\
	\sup_n \frac{1}{v_n} \log \vn{g_n} < \infty, \qquad \inf_{n} \inf_{x \in E_n} g_n(x) > - \infty, \label{eqn:convergence_condition_superlim_uniform_gn}
	\end{gather}
	and that, for every compact set $K \subseteq E$ and every sequence $z_n \in K \cap E_n$ satisfying $ \lim_n z_n = z$ and $\lim_n f_n(z_n) = f(z) > - \infty$, we have
	\begin{equation}\label{eqn:superlim_generators_lowerbound}
	\liminf_{n \uparrow \infty}g_{n}(z_n) \geq g_*(z).
	\end{equation}
	The \textit{extended limit}, denoted by $ex-\LIM_n H_n$ is defined as $ex-\subLIM_n H_n \cap ex-\superLIM_n H_n$.
\end{definition}

\begin{remark} \label{remark:LIM_via_simple_convergence}
	Suppose that $H \subseteq C_b(E) \times C_b(E)$ such that for all $(f,g) \in H$ there are $(f_n,g_n) \in H_n$ with $f = \LIM f_n$ and $g = \LIM g_n$, then $H \subseteq ex-\LIM_n H_n$.
\end{remark}

All our large deviation statements will be based on the following assumption.

\begin{assumption} \label{assumption:LDP_assumption}
	Let $E_n \subseteq E$ be Polish subsets satisfying $E = \lim_{n \rightarrow \infty} E_n$. 
	
	Assume that for each $n \geq 1$, we have $A_n \subseteq C_b(E_n) \times C_b(E_n)$ and existence and uniqueness holds for the $D_{E_n}(\bR^+)$ martingale problem for $(A_n,\mu)$ for each initial distribution $\mu \in \cP(E_n)$. Letting $\PR_{y}^n \in \cP(D_{E_n}(\bR^+))$ be the solution to $(A_n,\delta_y)$, the mapping $y \mapsto \PR_y^n$ is measurable for the weak topology on $\cP(D_{E_n}(\bR^+))$. Let $X_n$ be the solution to the martingale problem for $A_n$ and set
	\begin{equation*}
	H_n f = \frac{1}{r_n} e^{-r_nf}A_n e^{r_nf} \qquad e^{r_nf} \in \cD(A_n),
	\end{equation*}
	for some sequence of speeds $\{r_n\}_{n \geq 1}$, with $\lim_{n \rightarrow \infty} r_n = \infty$.  Let $\{\cH_J\}_{J \in \cJ(E)}$ be a generating set of Hamiltonians as in Definition \ref{definition:generating_Hamiltonians} and suppose that $H_\dagger \subseteq ex-\subLIM_n H_n$ and $H_\ddagger \subseteq ex-\superLIM_n H_n$. Note: by definition $C_c^2(E) \subseteq \cD(H_\dagger) \cap \cD(H_\ddagger)$.
\end{assumption}

\begin{definition}
	Let $r_n >0$ be a sequence of real numbers such that $r_n \rightarrow \infty$. We say that a sequence of processes $\{X_n\}_{n \geq 1}$ satisfies the \textit{exponential compact containment condition} at speed $\{r_n\}_{n \geq 1}$ if for every $T > 0$ and $a \geq 0$, there exists a compact set $K_{a,T} \subseteq E$ such that
	\begin{equation*}
	\limsup_{n \rightarrow \infty} \frac{1}{r_n} \log \PR\left[X_n(t) \notin K_{a,T} \text{ for some } t \leq T \right] \leq -a.
	\end{equation*}
\end{definition}

The following result follows along the lines of Proposition A.15 in \cite{CoKr17}, whose proof is based on Lemma 4.22 in \cite{FK06}.

\begin{proposition} \label{proposition:compact_containment}
	Consider the setting of Assumption \ref{assumption:LDP_assumption}. Suppose that $\{X_n(0)\}_{n \geq 1}$ is exponentially tight with speed $\{r_n\}_{n \geq 1}$.
	
	Suppose that $\Upsilon$ is a good containment function for $\{\cH_J\}_{J \in \cJ(E)}$. Then the processes $\{X_n\}_{n \geq 1}$ satisfy the exponential compact containment condition at speed $\{r_n\}_{n \geq 1}$
\end{proposition}

\begin{theorem} \label{theorem:abstract_LDP}
	Consider the setting of Assumption \ref{assumption:LDP_assumption}. Suppose that $X_n(0)$ satisfies a large deviation principle with speed $\{r_n\}_{n \geq 1}$ and good rate function $I_0$.
	\begin{enumerate}[(a)]
		\item Suppose that $\Upsilon$ is a good containment function for $\{\cH_J\}_{J \in \cJ(E)}$. Then the processes $\{X_n\}_{n \geq 1}$ are exponentially tight with speed $r_n$ in $D_E(\bR^+)$.
		\item In addition to the assumption in (a), suppose that for each $\lambda > 0$ and $h \in C_b(E)$ the comparison principle is satisfied for subsolutions to $f - \lambda H_\dagger f = h$ and supersolutions to $f - \lambda H_\ddagger f = h$. Then the large deviation principle is satisfied with speed $r_n$ for the processes $X_n$ with  good rate function $I$ 
		\begin{equation*}
		I(\gamma) = \begin{cases}
		I_0(\gamma(0)) + \int_0^\infty \cL(\gamma(s),\dot{\gamma}(s)) \, \dd s   & \text{if } \gamma \in \cA \cC(E), \\
		\infty & \text{otherwise}.
		\end{cases}
		\end{equation*}
	\end{enumerate}	
\end{theorem}

\begin{proof}[Proof of Theorem \ref{theorem:abstract_LDP}]	
	(a) follows from Proposition \ref{proposition:compact_containment} and Corollary 4.19 in \cite{FK06}. For the verification of condition(c) for the latter result, we use $F = C_c^2(E)$, whereas (d) is satisfied as $C_c^2(E) \subseteq \cD(H_\dagger)$ by Assumption \ref{assumption:LDP_assumption}.
	
	(b) follows from Theorem 8.27 and Corollary 8.28 in \cite{FK06}. For the application of this result, we use $\bfH_\dagger = H_\dagger$, $\bfH_\ddagger = H_\ddagger$ and $\cA f(x,v) = \ip{\nabla f(x)}{v}$. The conditions for the control theory part in \cite{FK06} have been verified in Propositions \ref{proposition:FK8.9} and \ref{proposition:FK8.11}. Note that the rate function in \cite{FK06} still involves an infimum over control measures. As our Lagrangian is convex in the speed variable, Jensen's inequality gives the final form.
\end{proof}

\section{Proofs for the examples} \label{section:proofs_examples}

We proceed by giving proofs for the results given in Section \ref{section:LDP_for_examples}. These proofs rely on Theorem \ref{theorem:abstract_LDP}, which means that for each set of examples we have to provide a containment function and verify the comparison principle for the corresponding set of Hamilton-Jacobi equations.

\subsection{One dimensional examples}

For our one-dimensional examples, the construction of a containment function will be case dependent. This will also hold for the verification of the comparison principle, although all these verifications will be carried out using Theorem \ref{theorem:comparison_principle_1d}.

\subsubsection{Birth and Death process with immigration, proof of Theorem \ref{theorem:birth_death_immigration}}

\begin{remark}
	Using the framework of Section \ref{section:framework_HJ_boundary}, we  have $E=\cB_{0,1}$, $\cJ=\{\partial_-,\emptyset\}$ and $E_\emptyset=E=[0,\infty)$, $E_{\partial_-}=\{\partial_-\}$ with $\partial_-=0$. For $x>0$, $T_xE=\bR$ and $T_{\partial_-}E=\bR^+$. 
\end{remark}

\begin{lemma} \label{lemma:containment_function_1d_immi}
	Consider the setting of Theorem \ref{theorem:birth_death_immigration}. Then there is a good containment function $\Upsilon$ for $\cH$.
\end{lemma}

\begin{proof}
	Let $\Upsilon : \bR^+ \rightarrow \bR^+$ be twice continuously differentiable with $\Upsilon(x)=\log(\log(x))$ for $x\geq e^e$ and $\Upsilon$ is zero somewhere in $[0,e^e)$. Plugging $\Upsilon'$ into $\cH$ for $x>e^e$ :
	\begin{multline*}
	\cH(x,\Upsilon'(x)) \leq (\lambda(x)+\rho(x))\left(e^{\frac{1}{x\log(x)}}-1\right)\\
	= \frac{\lambda(x)+\rho(x)}{x\log(x)}\left(\sum_{k=1}^{\infty}\frac{1}{k!}\frac{1}{(x\log(x))^{k-1}}\right) < \infty.
	\end{multline*}
\end{proof}

\begin{proof}[Proof of Theorem \ref{theorem:birth_death_immigration}]
	We apply Theorem \ref{theorem:abstract_LDP}. We start with the verification of Assumption \ref{assumption:LDP_assumption}. The space rescaled operators $A_n$ are given by
	\begin{equation*}
	A_nf(x)(x)=(\lambda_n(nx)+\rho_n(nx))\left[f\left(x+\frac{1}{n}\right) -f(x) \right]+\mu_n(nx)\left[f\left(x-\frac{1}{n}\right)-f(x)\right].
	\end{equation*}
	As a consequence the operators $H_n = \tfrac{1}{n}e^{-nf}A_ne^{nf}$ have the form
	\begin{equation*}
	H_nf(x)=\tfrac{1}{n}(\lambda_n(xn)+\rho_n(xn))\left[e^{n\left(f\left(x+\frac{1}{n}\right)-f(x)\right)}-1\right]+ \tfrac{1}{n}\mu_n(xn)\left[e^{n\left(f\left(x-\frac{1}{n}\right)-f(x)\right)}-1\right].
	\end{equation*}
	Taylor's theorem yields $H\subseteq ex-LIM_n H_n$ with $H$ given by $Hf(x) = \cH(x,f'(x))$ as in \eqref{equation:hamiltonian_birth_death} and domain $\cD(H)=C^2_c(\bR^+)$.  In particular, we have $H_\dagger = H_\ddagger = H$. 
	
	\smallskip
	
	The condition in (a) of Theorem \ref{theorem:abstract_LDP} follows by Lemma \ref{lemma:containment_function_1d_immi}. Thus, we are left to verify the condition for (b): the comparison principle. For this, we use Theorem \ref{theorem:comparison_principle_1d}, for which we need to verify Conditions  \ref{condition:basic} and \ref{condition:boundary}. 
	
	We start with the first one. Let $K \subseteq E^\circ$ be a compact set in the interior. Because both $(e^p-1)$ and $(e^{-p}-1)$ are bounded from below, we have $\lim_{p\rightarrow \pm\infty}\inf_{x\in K}\cH(x,p)=\infty$. This yields Condition \ref{condition:basic}.
	
	We proceed with Condition \ref{condition:boundary}.	Let $K\subseteq E$ a compact set with $\partial_{-}=0 \in K$. Because $\rho(0)>0$, we  have $\lim_{p\rightarrow +\infty}\inf_{x\in K}\cH(x,p)=\infty$. For $d=-$, we consider the weak condition (2). Let's pick $x_\alpha$ and $y_\alpha$ in $E$ both converging to $0$ such that $p_\alpha=\alpha(x_\alpha-y_\alpha)\rightarrow -\infty$. Then :
	\begin{align*}
	\cH(x_\alpha,p_\alpha)-\cH(y_\alpha,p_\alpha)\leq (\mu(x_\alpha)-\mu(y_\alpha))(e^{-p_\alpha}-1)\leq 0.
	\end{align*}
	Hence Condition \ref{condition:boundary} is satisfied and the comparison principle holds for the Hamilton-Jacobi equations for $H$. We conclude that the large deviation principle holds with Lagrangian rate function.
\end{proof}

\subsubsection{Growing populations, proof of Theorem \ref{theorem:LDP_growing_population}}

We first give a compact containment function.

\begin{lemma} \label{lemma:containment_function_H_positive_jumps}
	Consider the setting of Theorem \ref{theorem:LDP_growing_population}. Then there is a good containment function $\Upsilon$ for $\cH$.
\end{lemma}

\begin{proof}
	We construct a containment function for the Hamiltonian by piecing together two parts that control the behaviour at $0$ and at infinity respectively.
	
	\textbf{Control at $0$:} Pick $\Upsilon_0(x) = - \log x$. For $x \leq 1$, we have
	\begin{equation*}
	\cH(x,\Upsilon'_0(x)) = \sum_{k = 1}^\infty x v_k(x)  \left[e^{-kx^{-1}} - 1 \right] \leq 0.
	\end{equation*}
	\textbf{Control at $\infty$:} To control the behaviour at infinity, we consider Proposition \ref{proposition:compact_containment_at_infinity}. For this we rewrite the Hamiltonian as
	\begin{equation*}
	\cH(x,p) = \sum_{k=1}^\infty x v_k(x) \left[e^{kp}-1\right] = p \sum_{k=1}^\infty k v_k(x) x  + \sum_{k=1}^\infty x v_k(x) \left[e^{kp} - kp -1\right] 
	\end{equation*}
	so that $b(x) = \sum_{k=1}^\infty k v_k(x) x$ and $\nu(x,\cdot) = \sum_{k=1}^\infty v_k(x) x \delta_{\{k\}}$. Thus, the bounds on $b$ and $\nu$ in Proposition \ref{proposition:compact_containment_at_infinity} follow from Assumptions (b) and (c) on the functions $v_k$.
	
	We obtain that there is a function $\Upsilon_\infty$ such that $\sup_x H(x,\Upsilon_\infty'(x)) < \infty$.
	
	\textbf{Combining the bounds:} Let $\Upsilon : (0,\infty) \rightarrow \bR^+$ be a twice continuously differentiable function that is equal to $\Upsilon_0$ on the interval $(0,\frac{1}{2})$ and equal to $\Upsilon_\infty$ on $(1,\infty)$ and is equal to $0$ somewhere in $(\frac{1}{2},1)$.
	
	As $\Upsilon$ is twice continuously differentiable, $\cH$ is continuous and
	\begin{equation*}
	\sup_{x \in (0,\frac{1}{2}) \cup (1,\infty)} \cH(x,\Upsilon'(x)) < \infty
	\end{equation*}
	we find $\sup_x \cH(x,\Upsilon'(x)) < \infty$.
\end{proof}

\begin{proof}[Proof of Theorem \ref{theorem:LDP_growing_population}]
	We verify the conditions for Theorem \ref{theorem:abstract_LDP}. We start by checking Assumption \ref{assumption:LDP_assumption}.	First of all, the generator of the process $X_n(t)$ is given by
	\begin{equation*}
	A_nf(x) =	n \sum_{k=1}^\infty x v_k\left(nx\right) \left[f\left(x+\frac{k}{n}\right) - f(x)\right].
	\end{equation*}
	Thus,
	\begin{equation*}
	H_n f(x) = \frac{1}{n} e^{-nf(x)} \left(A_n e^{nf}\right)(x) = \sum_{k=1}^\infty x v_k\left(nx\right) \left[e^{n \left(f\left(x+\frac{k}{n}\right) - f(x)\right)} - 1\right].
	\end{equation*}
	and we obtain by \eqref{eqn:LDP_growing_population_convergence_condition} that $H \subseteq \LIM H_n$, where $H$ is the operator $\cD(H) = C_c^2(E)$, $Hf(x) = \cH(x,f'(x))$. We thus consider the setting that $H_\dagger = H_\ddagger = H$.
	
	\smallskip
	
	The condition in (a) of Theorem \ref{theorem:abstract_LDP} follows from Lemma \ref{lemma:containment_function_H_positive_jumps}. We thus proceed with the verification of the condition in (b).

	As before, we use Theorem \ref{theorem:comparison_principle_1d}. In this setting there is no boundary, thus, it suffices to verify Condition \ref{condition:basic} for the interior. Fix a compact set $K \subseteq (0,\infty)$. First consider $d = +$. Then assumption (a) on the functions $v_k$, we find $\lim_{p \rightarrow \infty} \inf_{x \in K} \cH(x,p) = \infty$, i.e. (1) of Condition \ref{condition:basic} is satisfied. For $d = -$, (2) of Condition \ref{condition:basic} is satisfied with $h^-(x) = - \sum_{k = 1}^\infty x v_k(x)$.
	
	Thus, the comparison principle holds by Theorem \ref{theorem:comparison_principle_1d} and the large deviation principle follows as a result.
\end{proof}

\subsubsection{SI model, proof of Theorem \ref{theorem:LDP_SI_model}}

\begin{lemma} \label{lemma:containment_function_SI_model}
	Consider the setting of Theorem \ref{theorem:LDP_SI_model}. Then there is a good containment function $\Upsilon$ for $\cH$.
\end{lemma}

\begin{proof}
	The only non-compactness issue arising for the SI model comes from the open boundary at $0$. Therefore the construction that was used in the proof Lemma \ref{lemma:containment_function_H_positive_jumps} to control the behaviour at $0$ can be used here as well.
\end{proof}

\begin{proof}[Proof of Theorem \ref{theorem:LDP_SI_model}]
	The proof is similar to that of Theorem \ref{theorem:LDP_growing_population}. The difference comes from the new boundary at $x=1$. We verify Condition \ref{condition:boundary} for the right boundary for the Hamiltonian $\cH(x,p) = c(x)\left[e^{p} - 1 \right]$. For $d = +$, i.e. momenta away from the boundary, we have for $K = [\tfrac{1}{2},1]$ that
	\begin{equation*}
	\lim_{p \rightarrow - \infty} \sup_{x \in K} \left| \cH(x,p) + c(x) \right| = 0.
	\end{equation*}
	For $d= -$, i.e. momenta towards the boundary, note that if $x_\alpha,y_\alpha \rightarrow 1$ and $p_\alpha := \alpha(x_\alpha - y_\alpha) \rightarrow \infty$, then $x _\alpha > y_\alpha$ for large $\alpha$. This implies that
	\begin{equation*}
	\cH(x_\alpha,p_\alpha) - \cH(y_\alpha,p_\alpha) = \left[c(x_\alpha) - c(y_\alpha)\right]\left[e^{p_\alpha}  - 1\right] \leq 0
	\end{equation*}
	as $c$ is decreasing in a neighbourhood of $0$.
\end{proof}

\subsubsection{Harvesting, proof of Theorem \ref{theorem:birth_death_immigration_harvesting}}

\begin{lemma} \label{lemma:containment_function_1d_immi_harvesting}
	Consider the setting of Theorem \ref{theorem:birth_death_immigration_harvesting}. Then there is a good containment function $\Upsilon$ for $\cH$.
\end{lemma}

\begin{proof}
	The containment function introduced in the proof of Lemma \ref{lemma:containment_function_1d_immi} works also in this setting. 
\end{proof}

\begin{proof}[Proof of Theorem \ref{theorem:birth_death_immigration_harvesting}]
	We verify the conditions for Theorem \ref{theorem:abstract_LDP}. We start with Assumption \ref{assumption:LDP_assumption}. The tilted operator $H_n$ for the process $t \mapsto \tfrac{1}{n} X_n(t)$ is given by:
	\begin{equation*}
	H_nf(x)=
	\begin{cases}
	\tfrac{1}{n}(\lambda_n(x) + \rho_n(x))\left[e^{f(x+\frac{1}{n})-f(x)}-1\right]+\tfrac{1}{n}(\mu_n(x) + \beta_n(x))\left[e^{f(x-\frac{1}{n})-f(x)} -1\right] & \text{if } x>0, \\
	\tfrac{1}{n}\mu_n(x) \left[e^{f(0+\frac{1}{n})-f(0)} -1\right] &  \text{if } x=0. \\
	\end{cases}
	\end{equation*}
	Taking point-wise limits, this yields two limiting Hamiltonians: $\cH_\emptyset$ and $\cH_{\partial_-}$ given by \ref{equation:hamiltonian_interior_harvest} and \ref{equation:hamiltonian_zero_harvest} respectively. The operators $H_\dagger$ and $H_\ddagger$ defined in Definition \ref{definition:generating_Hamiltonians} with domain $C^2_c(\bR^+)$ are then
	\begin{equation}
	\cH_\dagger(x,p)=
	\begin{cases}
	\cH_\emptyset(x,p)&\text{if }x>0\\
	\rho(0)\left[e^p-1\right]+(\beta(0)\left[e^{-p}-1\right]\vee 0) & \text{if }x=0\\
	\end{cases}
	\end{equation}
	\begin{equation}
	\cH_\ddagger(x,p)=
	\begin{cases}
	\cH_\emptyset(x,p)&\text{if }x>0\\
	\rho(0)\left[e^p-1\right]+(\beta(0)\left[e^{-p}-1\right]\wedge 0) & \text{if }x=0\\
	\end{cases}
	\end{equation}
	Next, we check that point-wise convergence can be strengthened to $H_\dagger \subseteq ex-\subLIM_n H_n$. Pick $f \in C_c^2(\bR^+)$ and write $f_n := f$ and $g_n := H_n f_n$. By our choice of $f_n,g_n$ \eqref{eqn:convergence_condition_sublim_uniform_gn} and \eqref{eqn:convergence_condition_sublim_constants} are immediate. We show that for $z_n \rightarrow z$ that  $\limsup_n g_n(z_n) \leq H_\dagger f(z)$.
	
	Clearly, if $z \neq 0$ then $g_n(z_n) \rightarrow H_\emptyset f(z) = H_\dagger f(z)$. If $z = 0$, then the values $g_n(z_n)$ depend on whether $z_n = 0$ or not. In either case, a limiting upper bound is given by the maximum of the individual limits: $H_\emptyset f(0) \vee H_{\partial_-} f(0) = H_\dagger f (0)$. We conclude that indeed $\limsup_n g_n(z_n) \leq H_\dagger f(z)$.
	
	As $f$ was arbitrary, we conclude $H_\dagger \subseteq ex-\subLIM_n H_n$. Similarly one proves $H_\ddagger \subseteq ex-\superLIM_n H_n$
	
	\smallskip
	
	All other steps of the proof follow that of Theorem \ref{theorem:birth_death_immigration}. The verification of the comparison principle changes slightly, for this we again use Theorem \ref{theorem:comparison_principle_1d} and verify Conditions \ref{condition:basic} and \ref{condition:boundary}.
	
	For the first condition of \ref{condition:basic} as well as \ref{condition:boundary}, note that as $\lambda, \mu > 0$ on the interior of $E$, $\rho$ and $\alpha$ are continuous and satisfy $\rho(0),\alpha(0) > 0$, we have for each compact set $K \subseteq E$ that
	\begin{equation*}
	\lim_{|p| \rightarrow \infty} \inf_{x \in K} \cH_\emptyset(x,p) = \infty.
	\end{equation*}
	
	For the second condition of \ref{condition:basic}, note that the boundary at $0$ is a left boundary, i.e. $s = -$. By \eqref{equation:hamiltonian_interior_harvest} and \eqref{equation:hamiltonian_zero_harvest}, we have
	\begin{equation*}
	\cH_\emptyset(0,p) - \cH_{\partial_-}(0,p) = \beta(0) \left[e^{-p} - 1 \right].
	\end{equation*}
	As $\beta(0) > 0$, the claim follows. We conclude that indeed Conditions \ref{condition:basic} and \ref{condition:boundary} are satisfied.
\end{proof}

\subsection{Proof of Theorem \ref{theorem:interacting_species}}

Before proving the result, we establish that there is a good containment function.

\begin{lemma} \label{lemma:containment_interacting_species}
	Consider the setting of Theorem \ref{theorem:interacting_species}. Then there is a good containment function for $\bH$.
\end{lemma}

\begin{proof}
	Let $s(x) = \sum x_i$ and let $\Upsilon$ be twice continuously differentiable and equal to $\log \log s(x)$ for $x$ with $s(x) \geq e^e$. The result then follows like Lemma \ref{lemma:containment_function_1d_immi}.
\end{proof}

\begin{proof}[Proof of Theorem \ref{theorem:interacting_species}]
	The result follows as in the proof of Theorem \ref{theorem:birth_death_immigration_harvesting}, using the containment function of Lemma \ref{lemma:containment_interacting_species} and the comparison principle obtained in Theorem \ref{theorem:comparison_principle_multid}.
\end{proof}

\appendix

\section{Proof of Proposition \ref{proposition:comparison_conditions_on_H} and Lemma \ref{lemma:control_on_H}} \label{appendix:abstract_comparison_verification}

The proof of Proposition \ref{proposition:comparison_conditions_on_H} is a variant of Proposition A.9 in \cite{CoKr17}, itself inspired a combination of Lemma 9.3 in \cite{FK06} and Lemma 2.3 in \cite{DFL11}. We reprove this result here as the boundary conditions introduces some additional issues involving the containment function that can be taken care of using the following lemma.

\begin{lemma} \label{lemma:elementary_min_bound}
	Let $a_1,\dots,a_k$ and $\delta_1,\dots,\delta_k$ be constants in $\bR$. Then we have
	\begin{equation*}
	\min_i \left\{a_i + \delta_i \right\} \leq \min_i a_i + \max_i \delta_i
	\end{equation*}
\end{lemma}

\begin{proof}
	Let $j,l$ be such that $a_j + \delta_j = \min_i \left\{a_i + \delta_i \right\}$ and $a_l = \min_i a_i$.
	Then, we have $\min_i \left\{a_i + \delta_i \right\} = a_j + \delta_j  \leq a_l + \max_i \delta_i  = \min_i a_i + \max_i \delta_i$.
\end{proof}

\begin{proof}[Proof of Proposition \ref{proposition:comparison_conditions_on_H}]
	Following the proof of Lemma A.8 in \cite{CoKr17}, we can extend the operators $H_\dagger,H_\ddagger$ to contain functions of the type
	\begin{equation*}
	x  \mapsto (1-\varepsilon) \alpha \Psi(x,y) + \varepsilon \Upsilon(x) + c, \qquad y \mapsto -(1+\varepsilon) \alpha \Psi(x,y) - \varepsilon \Upsilon(x) + c,
	\end{equation*}
	with $c \in \bR$ and $\alpha,\varepsilon > 0$, respectively. Let $x_{\alpha,\varepsilon},y_{\alpha,\varepsilon} \in E$ such that  \eqref{eqn:comparison_principle_proof_choice_of_sequences} is satisfied. Then, for all $\alpha$ we obtain that
	\begin{align}
	& \sup_x u(x) - v(x) \notag\\
	& = \lim_{\varepsilon \rightarrow 0} \sup_x \frac{u(x)}{1-\varepsilon} - \frac{v(x)}{1+\varepsilon} \notag\\
	& \leq \liminf_{\varepsilon \rightarrow 0} \sup_{x,y} \frac{u(x)}{1-\varepsilon} - \frac{v(y)}{1+\varepsilon} -  \alpha\Psi(x,y) - \frac{\varepsilon}{1-\varepsilon} \Upsilon(x) - \frac{\varepsilon}{1+\varepsilon}\Upsilon(y) \notag\\
	& = \liminf_{\varepsilon \rightarrow 0} \frac{u(x_{\alpha,\varepsilon})}{1-\varepsilon} - \frac{v(y_{\alpha,\varepsilon})}{1+\varepsilon} - \alpha \Psi(x_{\alpha,\varepsilon},y_{\alpha,\varepsilon}) - \frac{\varepsilon}{1-\varepsilon}\Upsilon(x_{\alpha,\varepsilon}) -\frac{\varepsilon}{1+\varepsilon}\Upsilon(y_{\alpha,\varepsilon}) \notag \\
	& \leq \liminf_{\varepsilon \rightarrow 0} \frac{u(x_{\alpha,\varepsilon})}{1-\varepsilon} - \frac{v(y_{\alpha,\varepsilon})}{1+\varepsilon}, \label{eqn:basic_inequality_on_sub_super_sol}
	\end{align}
	as $\Upsilon$ and $\Psi$ are non-negative functions. Since $u$ is a sub-solution to $f - \lambda H_\dagger f = h$ and $v$ is a super-solution to $f - \lambda H_\ddagger f = h$, we find by our particular choice of $x_{\alpha,\varepsilon}$ and $y_{\alpha,\varepsilon}$ that
	\begin{align}
	& u(x_{\alpha,\varepsilon}) - \lambda \cH_\dagger\left(x_{\alpha,\varepsilon}, (1-\varepsilon)\alpha\nabla \Psi(\cdot,y_{\alpha,\varepsilon})(x_{\alpha,\varepsilon}) + \varepsilon \nabla \Upsilon(x_{\alpha,\varepsilon})\right) \leq h(x_{\alpha,\varepsilon}), \label{eqn:ineq_comp_proof_1}\\
	& v(y_{\alpha,\varepsilon}) - \lambda \cH_\ddagger\left(y_{\alpha,\varepsilon},-(1+\varepsilon)\alpha\nabla \Psi(x_{\alpha,\varepsilon},\cdot)(y_{\alpha,\varepsilon}) - \varepsilon \nabla \Upsilon(y_{\alpha,\varepsilon})\right) \geq h(y_{\alpha,\varepsilon}).\label{eqn:ineq_comp_proof_2}
	\end{align}
	For all $z \in E$, the map $p \mapsto \cH_\dagger(z,p)$, being the maximum of convex maps, is convex. Thus, \eqref{eqn:ineq_comp_proof_1} implies that
	\begin{equation} \label{eqn:ineq_comp_proof_1_convexity}
	\begin{aligned}
	u(x_{\alpha,\varepsilon}) & \leq h(x_{\alpha,\varepsilon}) + (1-\varepsilon) \lambda \cH_\dagger(x_{\alpha,\varepsilon}, \alpha\nabla \Psi(\cdot,y_{\alpha,\varepsilon})(x_{\alpha,\varepsilon})) + \varepsilon \lambda \cH_\dagger(x_{\alpha,\varepsilon},\Upsilon'(x_{\alpha,\varepsilon}))   \\
	& \leq h(x_{\alpha,\varepsilon}) + (1-\varepsilon) \lambda H_\dagger(x_{\alpha,\varepsilon}, \alpha\nabla \Psi(\cdot,y_{\alpha,\varepsilon})(x_{\alpha,\varepsilon})) + \varepsilon \lambda c,
	\end{aligned}
	\end{equation}
	where $c := \max_i \sup_x \cH_i(x,\Upsilon'(x)) <\infty$ by the uniform bound ($\Upsilon$d).
	
	As far as the second inequality, first note that because $\Psi$ is a good penalization function, we have $- ( \nabla \Psi(x_{\alpha,\varepsilon},\cdot))(y_{\alpha,\varepsilon}) = \nabla \Psi(\cdot, y_{\alpha,\varepsilon})(x_{\alpha,\varepsilon})$. Next, we need a more sophisticated bound using the convexity of the maps $\{\cH_J\}_{J \in \cJ(E)}$. First of all, for each $J$ the convexity of $\cH_J$ yields
	\begin{multline*}
	\cH_J(y_{\alpha,\varepsilon},\alpha \nabla \Psi(\cdot, y_{\alpha,\varepsilon})(x_{\alpha,\varepsilon})) \\
	\leq \frac{1}{1+\varepsilon} \cH_J(y_{\alpha,\varepsilon},(1+\varepsilon)\alpha\nabla \Psi(\cdot, y_{\alpha,\varepsilon})(x_{\alpha,\varepsilon}) - \varepsilon \nabla \Upsilon(y_{\alpha,\varepsilon})) + \frac{\varepsilon}{1+\varepsilon} \cH_J(y_{\alpha,\varepsilon}, \nabla \Upsilon(y_{\alpha,\varepsilon})). 
	\end{multline*}
	This implies, using Lemma \ref{lemma:elementary_min_bound}, that 
	\begin{align*}
	& \cH_\ddagger(y_{\alpha,\varepsilon},\alpha \nabla \Psi(\cdot, y_{\alpha,\varepsilon})(x_{\alpha,\varepsilon})) \\
	& \leq  \min_{J:  y_{\alpha,\varepsilon} \in E_J} \left\{ \frac{1}{1+\varepsilon} \cH_J(y_{\alpha,\varepsilon},(1+\varepsilon)\alpha\nabla \Psi(\cdot, y_{\alpha,\varepsilon})(x_{\alpha,\varepsilon}) - \varepsilon \nabla \Upsilon(y_{\alpha,\varepsilon}))  \right. \\
	& \hspace{20em} + \left. \frac{\varepsilon}{1+\varepsilon} \cH_J(y_{\alpha,\varepsilon}, \nabla \Upsilon(y_{\alpha,\varepsilon})) \right\} \\
	& \leq \frac{\varepsilon}{1+\varepsilon}\cH_\ddagger(y_{\alpha,\varepsilon},(1+\varepsilon)\alpha\nabla \Psi(\cdot, y_{\alpha,\varepsilon})(x_{\alpha,\varepsilon}) - \varepsilon \nabla \Upsilon(y_{\alpha,\varepsilon})) + \frac{\varepsilon}{1+\varepsilon} c.
	\end{align*}
	Thus, \eqref{eqn:ineq_comp_proof_2} gives 
	\begin{equation} \label{eqn:ineq_comp_proof_2_convexity}
	v(y_{\alpha,\varepsilon}) \geq h(y_{\alpha,\varepsilon}) + \lambda (1+\varepsilon) \cH_\ddagger(y_{\alpha,\varepsilon},\alpha\nabla\Psi(\cdot,y_{\alpha,\varepsilon})(x_{\alpha,\varepsilon})) - \varepsilon \lambda c.
	\end{equation}
	By combining \eqref{eqn:basic_inequality_on_sub_super_sol} with \eqref{eqn:ineq_comp_proof_1_convexity} and \eqref{eqn:ineq_comp_proof_2_convexity}, we find
	\begin{align} 
	& \sup_x u(x) - v(x) \nonumber\\
	& \leq \liminf_{\varepsilon \rightarrow 0} \liminf_{\alpha \rightarrow \infty} \left\{ \frac{h(x_{\alpha,\varepsilon})}{1 - \varepsilon} - \frac{h(y_{\alpha,\varepsilon})}{1+\varepsilon} \right.  \label{eqn:eqn:comp_proof_final_bound:line1}\\
	& \qquad + \lambda  \left( \frac{\varepsilon}{1-\varepsilon} + \frac{\varepsilon}{1+\varepsilon} \right) c  \label{eqn:eqn:comp_proof_final_bound:line2}\\
	& \left. \qquad +  \lambda \left[\cH_\dagger(x_{\alpha,\varepsilon},\alpha\nabla\Psi(\cdot,y_{\alpha,\varepsilon})(x_{\alpha,\varepsilon})) - \cH_\ddagger(y_{\alpha,\varepsilon},\alpha\nabla\Psi(\cdot,y_{\alpha,\varepsilon})(x_{\alpha,\varepsilon}))\right] \vphantom{\sum} \right\}.\label{eqn:eqn:comp_proof_final_bound:line3}
	\end{align}
	The term \eqref{eqn:eqn:comp_proof_final_bound:line3} vanishes by assumption, whereas \eqref{eqn:eqn:comp_proof_final_bound:line2} vanishes as $\varepsilon \downarrow 0$. Now observe that, for fixed $\varepsilon$ and varying $\alpha$, the sequence $(x_{\alpha,\varepsilon},y _{\alpha,\varepsilon})$ takes its values in a compact set and hence admits converging subsequences. All these subsequences converge to points of the form $(z,z)$. Therefore, as $\alpha \rightarrow \infty$, we find
	\[
	\liminf_{\varepsilon \rightarrow 0} \liminf_{\alpha \rightarrow \infty}  \frac{h(x_{\alpha,\varepsilon})}{1 - \varepsilon} - \frac{h(y_{\alpha,\varepsilon})}{1+\varepsilon} \leq \liminf_{\varepsilon \rightarrow 0} \vn{h} \frac{2\varepsilon}{1-\varepsilon^2} = 0,
	\]
	giving that also the term in \eqref{eqn:eqn:comp_proof_final_bound:line1} converges to zero.

	\smallskip
	
	We conclude that the comparison principle holds for $f - \lambda H f = h$.
\end{proof}

\begin{proof}[Proof of Lemma \ref{lemma:control_on_H}]
	Using that $v$ is a super-solution to $f - \lambda H_\ddagger f = h$, we find that it is a super solution to the equation $f - \lambda \hat{H}_\ddagger f = h$, where $\hat{H}_\ddagger$ is super-extension, as above, of $H_\dagger$ that includes functions of the type $y \mapsto (-(1+\varepsilon)\alpha \Psi(x,y)- \varepsilon \Upsilon(y)$ in its domain, cf. Lemma A.8 of \cite{CoKr17}. It follows that for the points $(x_{\alpha,\varepsilon},y_{\alpha,\varepsilon})$, we have
	\begin{equation*}
	\cH_\ddagger(y_{\alpha,\varepsilon},(1+\varepsilon)\alpha \nabla \Psi(\cdot, y_{\alpha,\varepsilon})(x_{\alpha,\varepsilon}) - \varepsilon \nabla \Upsilon(y_{\alpha,\varepsilon})) \leq \frac{v(y_{\alpha,\varepsilon}) - h(y_{\alpha,\varepsilon})}{\lambda} \leq  \frac{\vn{v-h}}{\lambda}.
	\end{equation*}
	Similar to the use of the convexity of the functions $\cH_J$ in the proof of Proposition \ref{proposition:comparison_conditions_on_H}, we find
	\begin{multline*}
	\cH_\ddagger(y_{\alpha,\varepsilon}, \alpha \nabla \Psi(\cdot, y_{\alpha,\varepsilon})(x_{\alpha,\varepsilon})) \\
	\leq \frac{1}{1+\varepsilon} \cH_\ddagger(y_{\alpha,\varepsilon},(1+\varepsilon)\alpha\nabla \Psi(\cdot, y_{\alpha,\varepsilon})(x_{\alpha,\varepsilon}) - \varepsilon \nabla \Upsilon(y_{\alpha,\varepsilon})) \\
	+ \frac{\varepsilon}{1+\varepsilon} \max_J \cH_J(y_{\alpha,\varepsilon}, \nabla \Upsilon(y_{\alpha,\varepsilon}))
	\end{multline*}
	which implies
	\begin{equation*}
	\sup_{\alpha} \cH_\ddagger\left(y_{\alpha,\varepsilon},\alpha (\nabla \Psi(\cdot,y_{\alpha,\varepsilon}))(x_{\alpha,\varepsilon})\right) \leq \frac{1}{1+\varepsilon}\left( \frac{\vn{v-h}}{\lambda} + \varepsilon \sup_{z,J} \cH_J(z,\nabla \Upsilon(z))\right) < \infty.
	\end{equation*}
	
\end{proof}

\section{A general method to constructing containment functions for the boundary at infinity} \label{appendix:containment_function}

The following result has been proven in Example 4.23 in \cite{FK06}.

\begin{proposition}  \label{proposition:compact_containment_at_infinity}
	Let $G \subseteq \bR$ be some subset and let $\cH : G \times \bR \rightarrow \bR$ defined by
	\begin{equation*}
	\cH(x,p) = \frac{1}{2} a(x) p^2 + b(x) p + \int \left[e^{kp} - kp - 1\right] \nu(x,\dd k).
	\end{equation*}
	Suppose there exists $\alpha > 0$ such that
	\begin{equation} \label{eqn:compact_containment_integral_bound}
	\sup_x \int \left| \frac{z}{1+|x|}\right|^2 e^{\alpha |z|/(1+|x|)} \nu(x,\dd z) < \infty
	\end{equation}
	and some $C > 0$ such that
	\begin{equation*}
	x b(x) \leq C(1+|x|^2), \qquad |a(x)| \leq C (1+|x|^2).
	\end{equation*}
	Then there is some $\delta > 0$ such that if $\Upsilon(x) := \delta \log (1+x^2)$, then
	\begin{equation*}
	\sup_{x \in G} \cH(x,\Upsilon'(x)) < \infty.
	\end{equation*}
\end{proposition}

\section{Convex analysis} \label{appendix:convex_analysis}

\begin{definition}
	Let $\phi : \bR^d \rightarrow (-\infty,\infty]$ be a convex function. We write
	\begin{equation*}
	\partial \phi(p) := \left\{v \in \bR^d \, \middle| \, \phi(p') \geq \phi(p) + \ip{p' - p}{v} \right\},
	\end{equation*}
	for the (set-valued) sub-differential of $\phi$ at $p$.
\end{definition}

\begin{lemma}[Theorem 23.4 in \cite{Ro70}] \label{lemma:empty_subdifferential}
	If $\phi(p) = \infty$ then $\partial \phi(p) = \emptyset$.
\end{lemma}

\begin{proposition}[Corollary 23.5.1 in \cite{Ro70}] \label{proposition:convex_conjugate_subdiff_inverse}
	Let $\phi$ be a lower semi-continuous function, and $\phi^*$ its Legendre transform. Then $v \in \partial \phi(p)$ if and only if $p \in \partial \phi^*(v)$.
\end{proposition}

\section{Differential inclusions} \label{appendix:differential_inclusions}

We follow \cite{De92,Ku00}. Let $D \subseteq \bR^d$ be a non-empty set. A \textit{multi-valued mapping} $F : D \rightarrow 2^{\bR^d} \setminus \{\emptyset\}$ is a map that assigns to every $x \in D$ a set $F(x) \subseteq \bR^d$, $F(x) \neq \emptyset$.

\begin{definition}
	Let $I \subseteq \bR$ be an interval with $0 \in I$, $D\subseteq \bR^d$, $x \in D$ and $F : D \rightarrow 2^{\bR^d} \setminus \emptyset$ a multi-valued mapping. A function $\gamma$ such that
	\begin{enumerate}[(a)]
		\item $\gamma : I \rightarrow D$ is absolutely continuous,
		\item $\gamma(0) = x$,
		\item $\dot{\gamma}(t) \in F(\gamma(t))$ for almost every $t \in I$
	\end{enumerate}
	is  called a \textit{solution of the differential inclusion} $\dot{\gamma} \in F(\gamma)$ a.e., $\gamma(0) = x$.
\end{definition}

If we assume sufficient regularity on the multi-valued mapping $F$, we can ensure the existence of a solution to differential inclusions that remain inside $D$.

\begin{definition}
	Let $D \subseteq \bR^d$ be a non-empty set and let $F : D \rightarrow 2^{\bR^d} \setminus \{\emptyset\}$ be a multi-valued mapping.
	\begin{enumerate}[(i)]
		\item We say that $F$ is \textit{closed, compact or convex valued} if each set $F(x)$, $x \in D$ is closed, compact or convex, respectively.
		\item We say that $F$ is \textit{upper semi-continuous at $x \in D$} if for each neighbourhood $\cU$ of $F(x)$, there is a neighbourhood $\cV$ of $x$ in $D$ such that $F(\cV) \subseteq \cU$. 
		We say that $F$ is \textit{upper semi-continuous} if it is upper semi-continuous at every point.
	\end{enumerate}
\end{definition}

\begin{definition}
	Let $D \subseteq \bR^d$ be a closed non-empty set. The tangent cone to $D$ at $x$ is
	\begin{equation*}
	T_D(x) := \left\{z \in \bR^d \, \middle| \, \liminf_{\lambda \downarrow 0} \frac{d(y + \lambda z, D)}{\lambda} = 0\right\}.
	\end{equation*}
	The set $T_D(x)$ is sometimes called the  the \textit{Bouligand cotingent cone}.
\end{definition}

\begin{lemma}[Theorem 2.2.1 in \cite{Ku00}, Lemma 5.1 in \cite{De92}] \label{lemma:solve_differential_inclusion}
	Let $D \subseteq \bR^d$ be closed and let $F : D \rightarrow 2^{\bR^d} \setminus \{\emptyset\}$ satisfy
	\begin{enumerate}[(a)]
		\item $F$ has closed convex values and is upper semi-continuous;
		\item $F$ is upper semi-continuous;
		\item for every $x$, we have $F(x) \cap T_D(x) \neq \emptyset$;
		\item $F$ has bounded growth: there is some $c > 0$ such that $\vn{F(x)} = \sup\left\{|z| \, \middle| \, z \in F(x) \right\} \leq c(1 + |x|)$ for all $x \in D$.
	\end{enumerate}
	Then the differential inclusion $\dot{\gamma} \in F(\gamma)$ has a solution on $\bR^+$ for every starting point $x \in D$.
\end{lemma}

\smallskip

\textbf{Acknowledgement} 

The authors thank an anonymous referee for careful reading of the manuscript and suggestions that improved the text. The authors also thank Jean-Ren\'{e} Chazottes for numerous helpful discussions and for reading various drafts of this paper. 
A month-long stay of RK at the École Polytechnique in Paris during which most of the work was carried out was supported by Le Centre National de la Recherche Scientifique(CNRS). RK was partly supported by the Deutsche Forschungsgemeinschaft (DFG) via RTG 2131 High-dimensional Phenomena in Probability – Fluctuations and Discontinuity.

\bibliographystyle{abbrv}
\bibliography{../KraaijBib}
\end{document}